\documentclass[11pt]{amsart}
\usepackage{amsmath,mathabx}
\usepackage{amssymb}
\usepackage{amsthm} 
\usepackage{bbm,url}
\usepackage{graphicx}
\usepackage{fullpage}
\usepackage{amssymb}

\usepackage[usenames,dvipsnames,svgnames,table]{xcolor}

\usepackage{float}
\usepackage{caption}
\usepackage{subcaption}
\usepackage{amsopn}

\def\R{{\mathbb{R}}}
\def\E{{\mathbb{E}}}

\DeclareMathOperator*{\argmin}{arg\,min}
\newcommand{\eps}{\varepsilon}

\def\u{{\mathbf{u}}}
\def\v{{\mathbf{v}}}
\def\z{{\mathbf{z}}}

\def\y{{\mathbf{y}}}

\def\X{{\mathbf{X}}}
\def\S{{\mathbf{S}}}

\def\Z{{\mathbf{Z}}}
\def\D{{\mathbf{D}}}
\def\U{{\mathbf{U}}}
\def\V{{\mathbf{V}}}
\def\K{{\mathbf{K}}}

\def\W{{\mathbf{W}}}

\def\A{{\mathbf{A}}}
\def\B{{\mathbf{B}}}

\def\Y{{\mathbf{Y}}}
\def\I{{\mathbf{I}}}
\def\Q{{\mathbf{Q}}}
\def\M{{\mathbf{M}}}

\newcommand{\revised}[1]{\textcolor{black}{#1}}

\newcommand{\revtwo}[1]{\textcolor{black}{#1}}

\renewcommand{\epsilon}{\varepsilon}

\newtheorem{theorem}{Theorem}
\newtheorem{definition}{Definition}

\newtheorem{lemma}{Lemma}

\newtheorem{remark}{Remark}
\newtheorem{corollary}{Corollary}

\begin{document}

\title{Learning nonnegative matrix factorizations from compressed data} 

\author{Abraar Chaudhry, Elizaveta Rebrova} 
\maketitle
\begin{abstract}
 We propose a flexible and theoretically supported framework for scalable nonnegative matrix factorization.
The goal is to find nonnegative low-rank components directly from compressed measurements, accessing the original data only once or twice.
We consider compression through randomized sketching methods that can be adapted to the data, or can be oblivious.
We formulate optimization problems that only depend on the compressed data, but which can recover a nonnegative factorization which closely approximates the original matrix.
The defined problems can be approached with a variety of algorithms, and in particular, we discuss variations of the popular multiplicative updates method for these compressed problems.
We demonstrate the success of our approaches empirically and validate their performance in real-world applications.
\end{abstract}

\section{Introduction}
Low-rank approximations are arguably the main tool for simplifying and interpreting large, complex datasets. Methods based on singular value decomposition of the data matrix deliver deterministic, useful results via polynomial-time algorithms. However, for nonnegative data, spatial localization and interpretability of the features can be boosted by additionally making the factors element-wise nonnegative \cite{lee1999learning}. In the standard form, given an \revtwo{element-wise} nonnegative matrix $\X \in \R_{\revtwo{+}}^{m \times n}$, and a target rank $r$, the Nonnegative Matrix Factoriration (NMF) problem is the task of finding matrices $\U \in \R^{m \times r}_{\revtwo{+}}$ and $\V \in \R^{n \times r}_{\revtwo{+}}$ solving the problem

\vspace*{-.5cm}

\begin{equation}\label{eq:nmf}
    \min_{\U,\V\geq 0} \|\X-\U\V^T\|_F^2.
\end{equation}

\vspace*{-.1cm}
From the component matrices $\U$ and $\V$, one can obtain soft clustering of the data with additional interpretability of the components, compared to the SVD results \cite{lee1999learning}.
NMF became a standard analysis tool across many application areas, such as topic modeling and text mining \cite{chen2019experimental,vendrow2021guided}, image processing \cite{lee1999learning}, hyperspectral unmixing \cite{huang2019spectral,feng2022hyperspectral}, genomics \cite{shao2017robust}, and others. It is amenable to numerous extensions, such as, incorporating semi-supervision \cite{lee2009semi}, tree-like hierarchy of the clusters \cite{grotheer-etal-2020-covid,kuang2013fast}, or additional information about the domain \cite{vendrow2021guided, gillis2008nonnegative}.

The problem \eqref{eq:nmf} is in general NP-hard \cite{nmf_np_hard} to solve and \revised{does not possess a unique solution. However, under additional conditions, the uniqueness up to scaling and permutation holds; see, e.g., \cite[Chapter~4]{gillisbook}.
Further,} several iterative algorithms have been developed to solve the NMF problem, including multiplicative updates (MU), alternating least squares (ALS), alternating nonnegative least squares (ANLS), and hierarchical alternating least squares (HALS).
See \cite[Chapter~8]{gillisbook} for a discussion of these methods. 

When the size of the data matrix $\X$ is large, it is challenging or impossible to store it, and even harder to use it within the optimization algorithms searching for the NMF decomposition.
Note that the resulting factors $\U$ and $\V$ collectively contain only $r(n+m)$ entries which is much less than the total of $nm$ entries in $\X$ if $r \ll \min\{n,m\}$.
\revtwo{Thus, it is reasonable to ask if one can store and process the resulting factors, even when }processing the whole data is challenging. 

This motivates us to apply a \emph{sketch-and-solve} approach to the NMF problem. 
This means that we will first compute a linear function of our input matrix $\mathcal{L}(\X)$, known as a \emph{sketch}, and then second, find a good factorization $\U\V^T$ based only on the sketch and the linear function $\mathcal{L}$, known as the \emph{sketching map}.
If the size of the sketch is much smaller than $\X$, the second task may be significantly more practical when there are memory limitations.  The practice of using sketches is also applicable to matrix factorization problems in other settings such as when different entries of $\X$ are revealed over time. In certain streaming applications, it has been shown that there is little benefit to considering nonlinear sketching functions as opposed to \emph{linear} sketches \cite{li2014turnstile}.

A standard and well-developed application of the linear sketch-and-solve approach is for the simpler problem of linear regression \cite{sarlos2006improved, drineas2011faster, raskutti2016statistical}. Wide classes of random matrices, defined obliviously to the underlying data, can be used as linear sketching operators, and deterministic conditions related to the geometry preservation properties of the sketching operators were formulated \cite{mahoney2011randomized,woodruff2014sketching}. Another prominent example of the sketch-and-solve approach is randomized SVD algorithms. To find low-rank factorization of a given matrix from its sketched measurements, the sketch should retain spectral properties of the matrix rather than being data oblivious. In \cite{halko2011finding}, a simple procedure of forming refined data-adapted sketches via just one pass over the data -- a randomized rangefinder algorithm --  was proposed. 

\revtwo{Despite recent progress on many important problems using sketching, there has been little systematic work on adapting the sketching ideas
to NMF.} Here, we develop \revtwo{a} theory showing why and how the data-adapted sketches \revtwo{as in \cite{halko2011finding}} are useful in finding \emph{nonnegative} low-rank components; and we also consider the cases when random, independent from the data, sketches can be used. 

 One way to sketch a structured object -- in our case, a matrix -- is to vectorize it and use random linear map on the resulting vector. This includes standard compressive sensing approaches for low-rank matrix recovery such as in \cite{davenport2016overview, fawzi2015lower}. Another way (which is the focus of this work) is to consider sketches that take of the form of left and right matrix products with $\X$, e.g., \revised{$\A_1\X$ or $\X \A_2$ for appropriately sized matrices $\A_1$ and $\A_2$}. Column and row sketching have been used successfully for matrix factorization and approximation problems \cite{tropp17, fazel2008compressed, WOOLFE2008335}, and its higher order analogue modewise sketching was used to speed up tensor low-rank fitting \cite{iwen2021lower,haselby2023modewise}. An advantage of this approach is in compact sketching matrices $\A_1 \in \R^{k \times m}$ and $\A_2 \in \R^{n \times k}$ that contain $k(n + m)$ elements together, compared to the $kmn$ entries in a linear sketch that is applied to a vectorization of the matrix $\X$ in $\R^{mn}$. Another advantage is in preserving the matrix structure of the data throughout the sketching, which is crucial in this work for integrating the compressed matrices within learning algorithms, such as the multiplicative updates algorithm.

\subsection{Contributions and outline}
The idea to make NMF problem scalable through randomized sketching was considered earlier. \emph{In Section~\ref{sec:lit}}, we review relevant related work. What was missing in the prior works is the crucial link between the algorithmic outcomes of compressed problems and their fitness for the original problem. Establishing such connection is challenging partially due to the limited theoretical guarantees for the convergence of NMF algorithms (which is essentially inevitable \revised{due} to the NP-hardness of the problem). We approach this challenge in the following way: \emph{(1) we define the compressed problems such that we can provably compare their optimal solutions with the optimal solutions to the uncompressed problem, and (2) we propose efficient algorithms for these compressed problems.} Due to (1), this also means getting good solutions for the original problem.

\revtwo{We note that popular and practically efficient algorithms to solve the NMF problem, such as multiplicative updates algorithm (MU) studied here, do not have strong convergence guarantees, neither on a standard NMF problem nor -- similarly -- on its compressed counterpart. Actually, both the algorithmic guarantees for the sketched and for the full NMF problem are developed similarly. Consecutively, the rigorous link between the compressed and full NMF is obtained in terms of the optimal values, not algorithmic outcomes. }
\vspace*{0.2cm}

\emph{In Section~\ref{sec:compressed-problems}}, we formulate optimization problems which depend only on sketches, but whose optimal solutions approximately solve the original NMF problem. In addition, these problems are formulated to be amenable for efficient solvers. 
We propose three types of such problems: (I) for two-sided row and column sketching, (II) for one-sided sketching with orthogonal data-adapted measurement matrices and (III) with approximately orthogonal (e.g., random) data-oblivious measurement matrices.

 The proposed compressed problem with row and column sketching is as follows:
    \begin{align*}
 \text{(I)} \quad \tilde{\U}, \tilde{\V} = \argmin_{\U, \V \ge 0} &\|\A_1 (\X- \U\V^T)\|_F^2 + \|(\X-\U\V^T)\A_2\|_F^2 \\ &+ \lambda_1\|P_{\X \A_2 }^\perp \U\V^T\|_F^2 + \lambda_2\|\U\V^T P_{\X^T \A_1^T}^\perp\|_F^2 \quad \text{(in Theorem~\ref{thm:exact-two-sided})}
\end{align*}
\revtwo{for some $\lambda_1, \lambda_2 > 0$.}
Theorem~\ref{thm:exact-two-sided} guarantees that if $\X$ has an exact nonnegative factorization of rank $r$, then the solution to the problem above is also exact $\X = \tilde{\U} \tilde{\V}$ as long as the sketching dimension is at least $r$. Crucially, the matrices $\A_1$ and $\A_2$ can be generic.  We explain how to solve this problem in the sketched space despite the inclusion of the regularizer terms involving orthogonal projections. Empirically, as shown in Section~\ref{sec: experiments}, this problem can be employed in a simplified form with $\lambda_1 = \lambda_2 = 0$, and it is suitable for the data with approximately low nonnegative rank: if the sketching matrices are generic (in particular, not designed to approximate the range of the data), the two-sided method should be employed for tight recovery. 
 
 The one-sided sketching can be more convenient for some types of the data and also is more compact. Indeed, iteratively solving the two-sided  problem requires storing and using both $\A_1$ and $\A_2$, whereas using sketches on one side takes twice less space for the same amount of compression. The proposed one-sided compressed problems formulations are
\begin{equation*}
\text{(II)} \quad  \tilde{\U}, \tilde{\V} = \argmin_{\U, \V \ge 0}\left[ \|\A(\X-\U\V^T)\|^2_F +  \lambda \|P_{\A^T}^\perp \U\V^T\|^2_F\right] \quad \text{(in Theorem~\ref{thm:oneside_penalty}), or } 
\end{equation*}
\begin{equation*}
  \text{(III)} \quad  \tilde{\U}, \tilde{\V} = \argmin_{\U, \V \ge 0}\left[ \|\A(\X-\U\V^T)\|^2_F +  \lambda\|\U\V^T\|^2_F\right] \quad \text{(in Theorem~\ref{thm: apx orthogonal A}). }
\end{equation*} 

So, what is required from the sketching matrices to work successfully within one-sided compression, (II) or (III)? Theorem~\ref{thm:oneside_penalty} is stated for the sketching matrices with orthonormal rows: in this case, the regularizer in the form $P^\perp_{\A^T} \U \V^T = (\I - \revised{\A^T\A})\U \V^T$ can be conveniently incorporated in the efficient solvers. Otherwise we can use the simplified regularizer without projection operator, Theorem~\ref{thm: apx orthogonal A}, where the resulting loss depends on $\A$ being approximately orthogonal (to the extent easily satisfied by generic random matrices, as described in Remark~\ref{rem:apx-orth-discussion}). We note that in the one-sided cases, nonzero regularization $\lambda$ is crucial both theoretically \revised{and} empirically (Figure~\ref{fig:lambda}).

Informally, both Theorems~\ref{thm:oneside_penalty} and \ref{thm: apx orthogonal A} state that when we find $\tilde{\U}$ and $\tilde{\V}^T$ solving the compressed problems stated above, the error $\| \X - \tilde{\U}\tilde{\V}^T\|^2_F$ depends on (a) how well an existent (unknown) solution of rank $r$ fits the uncompressed problem $\|\X - \U \V^T\|_F^2$, (b) how well the sketching matrix approximates the range of the data $\|P^\perp_{\A^T} \X\|^2_F$, and (c) how close to orthonormal are the rows of $\A$. In particular, this explains and quantfies how \emph{orthogonal data-adapted measurements} (e.g., constructed via the randomized rangefinder algorithm \cite{halko2011finding}) are useful in finding nonnegative low-rank components. By Corollary~\ref{cor:exact-one-sided}, in this case, the solution of the compressed problem above is exact for the original $\X$ that admits exact rank $r$ decomposition with the sketch size $k$ slightly oversamples $r$. Compared to that, \emph{data-oblivious one-sided measurements} incur additional loss, both theoretically and empirically, however they can be advantageous for other reasons. For example, they can be re-generated when needed without any access to data and they do not require an additional pass over the data to form them.  

\vspace*{0.2cm}

\emph{In Section~\ref{sec:solving}}, we move from comparing the optimal values to solving the three compressed problems above. We formulate new variants of the multiplicative updates (MU) algorithm for each of them and show that the losses are nonincreasing during their iterations in Corollaries~\ref{cor:solving-two-sided}, \ref{cor:solving-one-sided}, and~\ref{cor:solving apx orthogonal} respectively. These corollaries follow from a more general result Theorem~\ref{thm: MU} formulated for a generic loss function with sketching. We also briefly discuss using projected gradient descent method on our compressed objectives. 

One special technical challenge for using MU on the compressed objectives is that random sketching violates nonnegativity of the compressed components, which is the property that ensures the validity of the step sizes used in the MU algorithm. To address this challenge, we further generalize the defined compressed objective functions to include small shifts of the form 
$\sigma \|\mathbf{1}_m^T (\X-\U\V^T)\|^2$, 
where $\mathbf{1}_m$ is a vector of all ones in $\R^m$. This results in corrections of the form $\A^T\A + \sigma \mathbf{1}_m\mathbf{1}_m^T$  that restore required nonnegativity inside the recovery algorithms (relevant theory is developed in \emph{Subsection~\ref{sec:nonnegativity}}).  

\vspace*{0.2cm}

\emph{In Section~\ref{sec: experiments}} we give some examples on real and synthetic datasets, in particular, showing successful performance of the proposed methods using about $5\%$ of the initial data.
Finally, \emph{in Section~\ref{sec:future}}, we conclude with some directions for future research.

\section{Related work on scalable NMF}\label{sec:lit}
A number of randomized methods were proposed to improve the scalability of NMF, most of them in the form of heuristics. 

First, several works propose iterative random sketching approach, which requires sampling new random sketching matrices as the algorithm progresses.
Such works include the methods involving random projection streams \cite{yahaya2020gaussian,yahaya2021random} that allow for Johnson-Lindenstrauss-type sketching matrices (random and oblivious to the data) but require multiple sketches and passes over the initial data.
Similarly, in \cite{qian2018dsanls,qian2020fast} 
\revtwo{ memory-efficient versions of } NMF \revtwo{apply iteratively updated sketching } matrices \revtwo{ to the original data,} as opposed to a \revtwo{single} compression. 
\emph{Our focus is on the setting that requires one or two passes over the original data in the preprocessing stage while the algorithm that searches for nonnegative factors works solely on the compressed data.}

In \cite{sharan2019compressed}, the factor matrices are additionally assumed to be sparse and they were recovered with compressed sensing techniques. \emph{We do not make additional assumptions on the structure of the factor matrices.}

Data-adapted sketching with randomized SVD techniques was used in \cite{erichson2018randomized} in the context of the hierarchical alternating least squares algorithm, although no theoretical justification of the proposed methods was given. Recently, these ideas were extended to a symmetric variant of the NMF problem in \cite{hayashi2024randomized}. Additionally, in \cite{zhou2012fast}, a randomized SVD approximation is integrated into alternating multiplicative updates in a way that saves space, also without theoretical guarantees. 

The two works most closely related to ours are \cite{wang2010efficient} and \cite{tepper2016compressed}.
Both of these papers derive compressed objective functions and seek to apply semi-NMF methods to iteratively improve a nonnegative factorization. A semi-NMF is a factorization in which one factor is entrywise nonnegative and the other factor is not constrained. 
Both papers apply the semi-NMF multiplicative updates from \cite{ding2008convex} and the latter also considers other update methods including updates based on the alternating direction method of multipliers (ADMM).
Although the updates of \cite{ding2008convex} do possess guarantees to not increase their corresponding semi-NMF objectives, neither \cite{wang2010efficient} nor \cite{tepper2016compressed} show whether these guarantees can be extended to their NMF objectives. So, the validity of the derived objective functions or the convergence of proposed iterative methods on the original NMF problem was not theoretically justified. \emph{A crucial motivation of this work is to create a connection between the algorithms working on the compressed problems and their performance with respect to the solution to the original problem.} We achieve this with new variants of the standard NMF algorithms (different from those in \cite{tepper2016compressed, wang2010efficient}) for the newly formulated compressed objective functions. 
We also provide some numerical comparison between the methods in Section~\ref{sec: experiments}. 

\subsection{Notation} Throughout the text, matrices and vectors are denoted with bold letters. We denote Frobenius matrix norm as $\|.\|_F$ and the spectral (operator) norm as $\|.\|$. The matrix $\X \in \R^{m \times n}_+$ means it is element-wise nonnegative, the same is denoted by $\X \ge 0$ when the size of the matrix is clear from the context. Element-wise positive and negative parts of vectors and matrices are denoted as $(\cdot)_+ = \max(\cdot,0)$ and $(\cdot)_- = -\min(\cdot,0)$ respectively. Further, $\circ$ denotes element-wise multiplication and $/$ denotes element-wise division. $P_\Z$ is the linear projection operator onto the column space of a tall matrix $\Z$, projection to the orthogonal complement is $P_\Z^\perp := \I - P_\Z$.

\section{Compressed problems with reliable solutions}\label{sec:compressed-problems}
We formulate optimization problems analogous to \eqref{eq:nmf}, which do not require storing the entire data matrix $\X$ and instead use sketched measurements. This is achieved by the use of carefully chosen regularization terms. In this section, we prove that the formulated problems are guaranteed to approximately solve the original NMF problem. In the next section, we show that the standard NMF solvers are easily extendable to these \revtwo{new sketched and regularized} objective functions.

\subsection{Two-sided compression}
First, we show that if a matrix has an exact low-rank nonnegative matrix factorization, then one can recover an exact factorization using linear  measurements on both sides.

\begin{theorem}\label{thm:exact-two-sided}
Suppose $\X$ has an exact nonnegative factorization $\X = \U_0\V_0^T$, where $\U_0 \in \R_+^{m \times r}$, $\V_0 \in \R_+^{n \times r}$ \revised{with} $r \leq \min\{n,m\}$.
Let $\A_1$ and $\A_2$ be matrices of sizes $\revised{k_1} \times m$ and $n \times \revised{k_2}$ respectively such that $\A_1\X\A_2$ \revised{has the same rank as $\X$}\footnote{Note that this condition holds \emph{generically}, i.e., for all but a (Lebesgue) measure-zero set of matrices \revised{as long as $k_1$ and $k_2$ are both at least the rank of $\X$}.}.
For any $\lambda_1, \lambda_2 > 0$, let
\begin{align}\label{eq:two-sided}
 \tilde{\U}, \tilde{\V} \revtwo{\; \in \;}  \argmin_{\U, \V \ge 0} \|\A_1 (\X- \U\V^T)\|_F^2 &+ \|(\X-\U\V^T)\A_2\|_F^2 \\ &+ \lambda_1\|P_{\X \A_2}^\perp \U\V^T\|_F^2 + \lambda_2\|\U\V^T P_{\X^T \A_1^T}^\perp\|_F^2, \nonumber
\end{align}
where $(\U, \V \ge 0)$ means $(\U \in \R_+^{m \times r}, \V \in \R_+^{n \times r})$. Then $\X = \tilde{\U}\tilde{\V}^T$.
\end{theorem}

\begin{remark}[Implementation considerations]\label{rem:implementation} This and further objective functions are formulated to allow for memory-efficient computations. 
\revised{For example, in the objective function~\eqref{eq:two-sided} above, it is not necessary to store the large matrix $\X$; one can instead store $\A_1\X$ and $\X\A_2$.}
Likewise, one does not need to store or compute $P_{\X\A_2}^\perp$ which is an $m \times m$ matrix, since one can instead compute
\begin{align*}
\|P_{\X\A_2}^\perp \U\V^T\|_F^2 &= \textrm{Tr}(\V\U^TP_{\X \A_2}^\perp \U\V^T) \\ 
&= \revised{\textrm{Tr}(\V\U^T(\I-P_{\X \A_2}) \U\V^T)}\\
&= \textrm{Tr}(\V\U^T \U\V^T) - \textrm{Tr}(\V\U^T \Q_2 \Q_2^T \U\V^T),
\end{align*}
where $\Q_2$ is an $m \times k_2$ matrix with columns that form the orthonormal basis of the columns of $\X \A_2$, and so $P_{\X \A_2} = \Q_2 \Q_2^T$. 
One can do a similar trick to compute $P_{\X^T\A_1^T}^\perp$ in terms of an analogous  matrix $\Q_1$\footnote{The efficient ways to find orthogonal bases of the column/row span are well-known, e.g., see the discussion in \cite{halko2011finding} Section 4.1.}.
Thus, one can compute the objective function of \eqref{eq:two-sided} with total storage cost \revised{$3r'(n+m)$, where $r' = \max\{r, k_1, k_2\}$}, by storing the matrices $\U, \V, \A_1\X, \X \A_2, \Q_1 , \Q_2$. This and similar considerations are crucial in Section~\ref{sec:solving}, when we study computationally efficient iterative algorithms  that solve the optimization problems on compressed data defined here.
\end{remark}

The proof of Theorem~\ref{thm:exact-two-sided} is loosely inspired by the row-column matrix sensing argument from \cite{fazel2008compressed}. 
\begin{proof} First, we show that the matrices $\U_0, \V_0$ such that $\X = \U_0 \V_0^T$ are not only feasible for the problem \eqref{eq:two-sided} but also give zero objective value.
Indeed, since $\A_1 \X \A_2$ \revised{has the same rank as $\X$, it follows that the matrices $\A_1 \X$ and $\X \A_2$ also have the same rank as $\X$.
This implies that the row space of $\A_1 \X$ and the column space of $\X \A_2$ are the same as the row space and column space of $\X$, respectively.
Thus we have $P_{\X \A_2} \X = P_{\X}$ and $P_{\X^T\A_1^T} \X = P_{\X^T}$.
Then one can observe that 
\[
\| \U_0\V_0^T P_{\X^T \A_1^T}^\perp \|_F = \|\U_0\V_0^T P_{\X^T}^\perp\|_F = \| \X P_{\X^T}^\perp \|_F = \| \mathbf{0} \|_F=0,
\]
}
and similarly $P_{\X \A_2}^\perp \X = \mathbf{0}$ and matrices $\U_0, \V_0$ give zero objective value.

Then, since $\tilde{\U},\tilde{\V}$ are optimal for \eqref{eq:two-sided}, they must also result in objective value $0$ and all four summands in \eqref{eq:two-sided} vanish:
\begin{equation}\label{eq:sum-12-vanish}
\A_1 (\X- \tilde \U \tilde \V^T) = (\X-\tilde \U\tilde \V^T)\A_2 = \mathbf{0},
\end{equation}
and 
\begin{equation}\label{eq:sum-34-vanish}
    P_{\X \A_2}^\perp \tilde \U \tilde \V^T = \tilde \U\tilde \V^T P_{\X^T \A_1^T}^\perp = \mathbf{0}.
\end{equation}

\revised{Let $\revtwo{\bar{\U}} \Sigma \revtwo{\bar{\V}}^T$ be a compact singular value decomposition of $\X$ (i.e., such that $\Sigma$ is invertible).
Observe that $P_{\X} = \revtwo{\bar{\U}} \revtwo{\bar{\U}}^T$ and $P_{\X^T} = \revtwo{\bar{\V}} \revtwo{\bar{\V}}^T$ and that the matrices $\A_1 \revtwo{\bar{\U}}$ and $\revtwo{\bar{\V}}^T \A_2$ have the same rank as $\X$.
Then we can write
\begin{equation}\label{two-sided-1} 
\tilde{\U}\tilde{\V}^T = \tilde{\U}\tilde{\V}^T P_{\X^T} = P_{\X}\tilde{\U}\tilde{\V}^T P_{\X^T} = \revtwo{\bar{\U}} \M \revtwo{\bar{\V}}^T, 
\end{equation}
where the matrix 
$$
\M := \revtwo{\bar{\U}}^T \tilde{\U}\tilde{\V}^T \revtwo{\bar{\V}} .
$$
Now observe 
$$
\revtwo{\bar{\U}}^T \A_1^T \A_1 \X \A_2 \A_2^T \revtwo{\bar{\V}} \overset{\eqref{eq:sum-12-vanish}}{=} 
\revtwo{\bar{\U}}^T \A_1^T \A_1 \tilde{\U}\tilde{\V}^T \A_2 \A_2^T \revtwo{\bar{\V}}  \overset{\eqref{two-sided-1}}{=}  \revtwo{\bar{\U}}^T \A_1^T \A_1 \revtwo{\bar{\U}} \M \revtwo{\bar{\V}}^T \A_2 \A_2^T \revtwo{\bar{\V}}.
$$
The rank of $\revtwo{\bar{\U}}^T \A_1^T \A_1 \revtwo{\bar{\U}}$ is the same as the rank of $\A_1 \revtwo{\bar{\U}}$ which is the same as the rank of $\X$, which is the same as the dimensions of $\revtwo{\bar{\U}}^T \A_1^T \A_1 \revtwo{\bar{\U}}$ (since the SVD is compact); thus, the matrix $\revtwo{\bar{\U}}^T \A_1^T \A_1 \revtwo{\bar{\U}}$ is invertible and likewise, the matrix $\revtwo{\bar{\V}}^T \A_2 \A_2^T \revtwo{\bar{\V}}$ is invertible.
Thus, we have
\begin{align*}
\M &= (\revtwo{\bar{\U}}^T \A_1^T \A_1 \revtwo{\bar{\U}})^{-1} \revtwo{\bar{\U}}^T \A_1^T \A_1 \X \A_2 \A_2^T \revtwo{\bar{\V}} (\revtwo{\bar{\V}}^T \A_2 \A_2^T \revtwo{\bar{\V}})^{-1}\\
&= (\revtwo{\bar{\U}}^T \A_1^T \A_1 \revtwo{\bar{\U}})^{-1} \revtwo{\bar{\U}}^T \A_1^T \A_1 \revtwo{\bar{\U}} \Sigma \revtwo{\bar{\V}}^T \A_2 \A_2^T \revtwo{\bar{\V}} (\revtwo{\bar{\V}}^T \A_2 \A_2^T \revtwo{\bar{\V}})^{-1}\\
&= \Sigma,
\end{align*}
and therefore we conclude that $\tilde{\U}\tilde{\V}^T = \revtwo{\bar{\U}} \M \revtwo{\bar{\V}}^T = \revtwo{\bar{\U}} \Sigma \revtwo{\bar{\V}}^T = \X$.}

\end{proof}

\revised{
\begin{remark}
Note that the conclusion of \ref{thm:exact-two-sided} holds for any matrices with zero objective value for problem \eqref{eq:two-sided}, regardless of the constraint set --- so, in particular, nonnegativity of the factors $\U_0$ and $\V_0$ is not used directly in the proof beyond the fact that the objective value of the uncompressed problem is zero.
\end{remark}
}

\subsection{One-sided compression: orthogonal sketching matrices}
Note that the described method requires measurements on both sides (and taking either $\A_1$ or $\A_2$ to be the identity matrix results in a necessity to work with the full matrix $\X$).
Now, we will show that it can be enough to measure the matrix on one side only. 
\begin{theorem}\label{thm:oneside_penalty} (Orthogonal $\A$)
Let $\X \in \R_+^{m \times n}$ be any matrix and let $\A \in \R^{k \times m}$ be a matrix with orthogonal rows (i.e., $\A\A^T = \I$). Let $\U_0 \in \R_+^{m \times r}$, $\V_0 \in \R_+^{r \times n}$ give a solution to the original NMF problem \eqref{eq:nmf} of rank $r$ and $\X_0 = \U_0\V_0^T$. \revtwo{Suppose}  $\tilde{\U},\tilde{\V}$  solve a compressed NMF problem with the same rank $r$, that is,
\begin{equation}\label{eq:comp-nmf}
    \tilde{\U}, \tilde{\V} \revtwo{\; \in \;} \argmin_{\U, \V \ge 0}\left[ \|\A(\X-\U\V^T)\|^2_F +  \lambda \|P_{\A^T}^\perp \U\V^T\|^2_F\right],
\end{equation}
where $\lambda > 0$, $P_{\A^T}^\perp := \I - \A^T\A$, and $(\U, \V \ge 0)$ means $(\U \in \R_+^{m \times r}, \V \in \R_+^{r \times n})$. Then $\tilde{\X} := \tilde{\U} \tilde{\V}^T$ satisfies
\begin{equation}\label{eq:one-sided-loss}\frac{\|\X - \tilde{\X}\|^2_F}{\|\X\|^2_F} \le c_\lambda \left[ \frac{\| \X - \X_0\|^2_F}{\|\X\|^2_F} + \frac{\|P^\perp_{\A^T} \X\|^2_F}{\|\X\|^2_F}\right],
\end{equation}
where $c_\lambda = \max(2/\lambda, 6, 2\lambda + 2).$
\end{theorem}

Before we start the proof, let us recall a simple corollary of the Pythagorean theorem and the triangle inequality to be used several times below: for any matrices $\X, \Y$ and a projection operator $P_{\A^T}$, 
\begin{align}\label{eq:triangle}
\|\X - \Y\|^2_F &\leq \|P_{\A^T}(\X - \Y)\|^2_F + 2\|P_{\A^T}^\perp \Y \|^2_F + 2\|P_{\A^T}^\perp \X \|^2_F.
\end{align}
Indeed, this follows from 
\begin{align*}
\|\X - \Y\|^2_F &=\|P_{\A^T}(\X - \Y)\|^2_F + \|P_{\A^T}^\perp(\X - \Y)\|^2_F \nonumber\\
&\leq \|P_{\A^T}(\X - \Y)\|^2_F + (\|P_{\A^T}^\perp \Y \|_F + \|P_{\A^T}^\perp \X \|_F)^2 \\
&\leq \|P_{\A^T}(\X - \Y)\|^2_F + 2\|P_{\A^T}^\perp \Y \|^2_F + 2\|P_{\A^T}^\perp \X \|^2_F.\nonumber
\end{align*}

\begin{proof}[Proof of Theorem~\ref{thm:oneside_penalty}]

First, note that
\begin{equation}\label{feas}
\|\A(\X-\tilde{\X})\|^2_F +  \lambda\|P_{\A^T}^\perp \tilde{\X}\|^2_F  \leq \|\A(\X-\X_0^T)\|^2_F +  \lambda\|P_{\A^T}^\perp \X_0\|^2_F,
\end{equation}
since $\tilde{\U},\tilde{\V}$ minimize the objective of \eqref{eq:comp-nmf} over all nonnegative matrices of the appropriate sizes. Now, since $\A\A^T = \I$, observe that  for any $\M \in \R^{m \times n}$ matrix $$\|\A\M\|_F = \|\A^T\A\M\|_F = \|P_{\A^T}\M\|_F.$$
So, using \revtwo{inequality} \eqref{eq:triangle} for the matrices $\X$ and $\tilde{\X}$, we can estimate
\begin{align*}
\|\X - \tilde{\X}\|^2_F 
 &\leq \|\A(\X - \tilde{\X})\|^2_F + 2\|P_{\A^T}^\perp \tilde{\X} \|^2_F + 2\|P_{\A^T}^\perp \X \|^2_F\\
&\leq c_1 \big( \|\A(\X - \tilde{\X})\|^2_F + \lambda \|P_{\A^T}^\perp \tilde{\X} \|^2_F \big) + 2\|P_{\A^T}^\perp \X \|^2_F\\
&\overset{\eqref{feas}}{\leq} c_1 \big( \|\A(\X - \X_0)\|^2_F + \lambda \|P_{\A^T}^\perp \X_0 \|^2_F \big) + 2\|P_{\A^T}^\perp \X \|^2_F
\end{align*}
for $c_1 = \max(2/\lambda, 1).$ Using \revtwo{inequality}  \eqref{eq:triangle} for the matrices $\X$ and $\X_0$, we can estimate the term in parentheses as
\begin{align*}
\|\A(\X - \X_0)\|^2_F &+ \lambda \|P_{\A^T}^\perp \X_0 \|^2_F  \\
&\le \|\A(\X - \X_0)\|^2_F + 2\lambda \|P_{\A^T}^\perp (\X - \X_0) \|^2_F + 2 \lambda \|P_{\A^T}^\perp \X \|^2_F \\
&\le c_2 \|\X - \X_0\|^2_F + 2 \lambda \|P_{\A^T}^\perp \X \|^2_F
\end{align*}
for $c_2 = \max(2\lambda,1).$ Combining the estimates and regrouping, we get
$$
\|\X - \tilde{\X}\|^2_F \le c_1c_2\|\X - \X_0\|^2_F + (2 \lambda c_1 + 2) \|P_{\A^T}^\perp \X \|^2_F.
$$
With $c_\lambda = \max (c_1c_2, 2 \lambda c_1 + 2) = \max(2/\lambda, 6, 2\lambda + 2)$, this concludes the proof of Theorem~\ref{thm:oneside_penalty}. 
\end{proof}

Theorem~\ref{thm:oneside_penalty} shows that the solution to the compressed NMF problem \eqref{eq:nmf} will work for the original uncompressed problem \eqref{eq:comp-nmf} as long as the terms $\| \X - \tilde{\X}\|^2_F$ and $\|P^\perp_{\A^T} \X\|^2_F$ are small. Luckily, with just one more pass over the original data one can get such sketching matrices using the standard approaches of randomized linear algebra, such as those in \cite{halko2011finding}.

\begin{theorem} [``Randomized rangefinder algorithm loss'', \cite{halko2011finding}]\label{thm:tropp}
Let $r,k$ be integers such that $r \geq 2$ and $r+2 \leq k \leq \min \{ m,n\}$.
Let $\X$ be an $m \times n$ matrix and $\S$ be a $n \times k$ standard Gaussian matrix.
Then 
\[
\mathbb{E} \|\X- P_{\X\S} \X\|_F \leq \left(1+\frac{r}{k-r-1}\right)^\frac{1}{2} \left(\sum_{j > r} \sigma_{j}^2(\X) \right)^\frac{1}{2},
\]
    where $\sigma_j(\X)$ is the $j$-th largest singular value of $\X$.
\end{theorem}

\begin{corollary}[Data-adapted one-sided sketches]\label{cor:exact-one-sided} Suppose the matrix $\X$ has an approximate nonnegative factorization, that is, there exist $\U_0 \in \R_+^{m \times r}$, $\V_0 \in \R_+^{r \times n}$ so that $\X_0 = \U_0\V_0^T$ satisfies $\|\X-\X_0\|_F \leq \varepsilon \|\X \|_F$. 

Take $k$ such that $2r + 1 \le k \le \min\{m,n\}$. Form a sketch $\X\S$ with $\S$ is an $n \times k$ standard Gaussian matrix; find $\Q$, \revtwo{an} orthonormal basis of the column space of $\X\S$; take a sketching matrix $\A := \Q^T$.  If $\tilde{\U} \in \R_+^{m \times r}, \tilde{\V} \in \R_+^{r \times n}$ solve a compressed NMF problem \eqref{eq:comp-nmf} with this $\A$ and some $\lambda > 0$, then 
\begin{equation}\label{eq:data-adapted-exact}
\frac{\E \|\X - \tilde{\X}\|_F}{\|\X\|_F} \le \sqrt{2} c_\lambda \varepsilon
\end{equation}
and $c_\lambda$ is the constant from \eqref{eq:one-sided-loss}. 
\end{corollary}
\begin{proof}
By Theorem~\ref{thm:tropp} and approximate low-rankness of $\X$, we have
\[
\mathbb{E} \|\X- P_{\A^T} \X\|_F \leq \sqrt{1+\frac{r}{k-r-1}} \sqrt{\sum_{j > r} \sigma_{j}(\X)} \le \sqrt{2} \|\X- \X_0\|_F \le  \sqrt{2} \varepsilon \|\X \|_F,
\]
using that $k \geq 2r+1$ in the second inequality.
Combining this with Theorem~\ref{thm:oneside_penalty}, we have
\[
\E \|\X - \tilde{\X}\|_F  \le 
 \revtwo{c_\lambda} (\| \X - \X_0\|_F + \E\|\X- P_{\A^T} \X\|_F) \leq \revtwo{c_\lambda} (1 + \sqrt{2}) \eps \|\X \|_F,
\]
where $\revtwo{c_\lambda} = \max(2/\lambda, 6, 2\lambda + 2).$
\end{proof}
\begin{remark}
It is easy to see that the oversampling condition $k > 2r$ can be relaxed to any $k > r + 1$ by suitably increasing the constant factor $\sqrt{2}$. Notwithstanding this factor, we see that if $\X$ has an exact NMF decomposition of rank $r$ and $k > r + 1$ then the error of the optimal solution to the compressed problem must be also zero, comparable with the result of Theorem~\ref{thm:exact-two-sided}. 

A high probability deviation bound for the loss $\|\X - \Q\Q^T \X\|_F$ is also known \cite[Theorem 10.7]{halko2011finding}. It implies a high probability estimate for \eqref{eq:data-adapted-exact} in a straightforward way. Instead of Gaussian initial sketching, one can employ subsampled random Fourier transform  \cite{halko2011finding} or other cost-efficient matrices \cite{saibaba2023randomized, ghashami2016frequent}. \revtwo{Moreover, similar results are available for the spectral norm instead of Frobenius (e.g., \cite[Theorem~10.3]{halko2011finding}), and the results of Theorem~\ref{thm:oneside_penalty} as well as Corollary~\ref{cor:exact-one-sided} extend to the spectral norm loss. An interesting question for the future work is how to extend similar results to the other standard losses appearing in the applications of the nonnegative matrix factorization problem, such as $\beta$-divergence \cite{fevotte2011algorithms} or $L_{2,1}$-norm loss \cite{kong2011robust}. }
\end{remark}

\subsection{One-sided compression: nonorthogonal sketching matrices} The orthogonality assumption on $\A$ can be relaxed to having approximately orthogonal rows, such as those of  appropriately normalized random matrices. This case is more than a straightforward extension of Theorem~\ref{thm:oneside_penalty} because of the following computational challenge: if the sketching matrix $\A$ does not have orthogonal rows, the orthogonal projection operator $P^\perp_{\A^T}$ does not have a nicely decomposable form $\A^T\A$. Theorem~\ref{thm: apx orthogonal A} below shows how to having projection matrices in the regularizer term.

\begin{definition} [Approximately orthogonal matrices]
For a positive constant $\eps < 1$, we call a matrix $\A \in \R^{k \times m}$ $\eps$-approximately orthogonal if its \revtwo{nonzero} singular values lie in the interval $[1-\eps,1+\eps]$.
\end{definition}
The convenience of the definition above stems from the following simple observation.
\begin{lemma}
If the matrix $\A \in \R^{k \times m}$ is $\eps$-approximately orthogonal, then for any $\M \in \R^{m \times n}$ matrix, we have 
\begin{equation}\label{eq:approx-orth}
(1 - \eps) \| P_{\A^T}\M\|_F \le \|\A\M\|_F \le (1 + \eps) \|P_{\A^T}\M\|_F.
\end{equation}
\end{lemma}
\begin{proof}
For a positive semidefinite matrix $\Z$, let $\sqrt{\Z}$ denote the unique positive semidefinite matrix such that $(\sqrt{\Z})^2 = \Z$. Then, if $\A = \U\boldsymbol{\Sigma} \V^T$ is a compact SVD decomposition of $\A$, $\sqrt{\A^T\A} =  \V \boldsymbol{\Sigma}\V^T$ and $P_{\A^T} = \V\V^T$. 
This implies
$
\|P_{\A^T} \M\|_F = \|\V^T\M\|_F, 
$ 
$\|\A\M\|_F 
= \| \sqrt{\A^T\A}\M \|_F  = \| \boldsymbol{\Sigma} \V^T\M \|_F$ and 
\begin{align*}
(1 - \eps)\|P_{\A^T}\M\|_F &\le ( 1 - \|\I - \boldsymbol{\Sigma}\|) \| \V^T\M \|_F\le \| \V^T\M \|_F  - \| (\I - \boldsymbol{\Sigma}) \V^T\M \|_F \\ &\le \| \boldsymbol{\Sigma} \V^T\M \|_F \le \|\boldsymbol{\Sigma}\|\|\V^T\M \|_F \le (1 + \eps) \|P_{\A^T}\M\|_F.
\end{align*}
\end{proof}

 The next theorem justifies solving a compressed NMF problem with a simplified regularization term:
\begin{theorem} \label{thm: apx orthogonal A}(Approximately orthogonal $\A$)
Let $\X \in \R_+^{m \times n}$ be any matrix and let $\A \in \R^{k \times m}$ be $\eps$-approximately orthogonal, with $\eps \leq 0.5$.
Let $\U_0 \in \R_+^{m \times r}$, $\V_0 \in \R_+^{r \times n}$ give a solution to the original NMF problem \eqref{eq:nmf} of rank $r$ and $\X_0 = \U_0\V_0^T$. If $\tilde{\U},\tilde{\V}$  solve the following compressed NMF problem with the same rank $r$
 \begin{equation}\label{comp-nmf-2}
     \tilde{\U},\tilde{\V} \revtwo{\; \in \;}  \argmin_{\U,\V \geq 0}\left[ \|\A(\X-\U\V^T)\|^2_F +  \lambda \| \U\V^T\|^2_F\right],
\end{equation}
 \revtwo{then} $\tilde{\X} := \tilde{\U}\tilde{\V}^T$ satisfies
\begin{equation}\label{eq:approx-orth-target-error}\frac{\left\|\X - \left( 1+ \lambda \right)\tilde{\X}\right\|^2_F}{\|\X\|^2_F} \leq c\left[\frac{\|\X - \revised{\X_0}\|^2_F}{\|\X\|^2_F} + \frac{\|P^\perp_{\A^T} \X\|^2_F}{\|\X\|^2_F} + \varepsilon^2 \right]
\end{equation}
where $c = \max(4+\frac{5}{4\lambda},6, 48\lambda)$.
\end{theorem}
\begin{proof}

By optimality, for any matrix $\Y = \U\V^T$ for some nonnegative $\U$ and $\V$ of the appropriate size (to be the scaled version of $\X_0$ as specified below) we have
\[
\|\A(\X-\tilde{\X})\|^2_F +  \lambda \| \tilde{\X}\|^2_F \leq \|\A(\X-\Y)\|^2_F +  \lambda \| \Y\|^2_F.
\]
Approximate orthogonality in the form of \eqref{eq:approx-orth} applied to the matrices $\M = \X -\tilde{\X}$ and $\M = \X -\Y$ allows to orthogonalize this inequality:
\[
(1-\varepsilon)^2 \|\Q(\X-\tilde{\X})\|^2_F + \lambda \| \tilde{\X}\|^2_F \leq (1+\varepsilon)^2 \|\Q(\X-\Y)\|^2_F +  \lambda \|\Y\|^2_F,
\]
where $\Q$ denotes the $k \times m$ matrix with orthogonal rows such that $P_{\A^T} = \Q^T\Q$. Indeed, this implies  
\begin{equation}\label{eq:q}
    \|P_{\A^T} \M\|_F = \|\Q^T\Q \M\|_F =  \|\Q \M\|_F
\end{equation}
for any $\M \in \R^{m \times n}.$
So, \begin{equation}\label{eq:approx-orth-mid-step}
\|\Q(\X-\tilde{\X})\|^2_F + \delta \|\tilde{\X}\|^2_F \leq \|\Q(\X-\Y)\|^2_F +  \delta \|\Y\|^2_F + 3\varepsilon^2\|\Q(\X-\Y)\|^2_F.
\end{equation}
with $\delta = \lambda/(1-\varepsilon)^2$.
\bigskip

We will further rearrange the optimality condition using the following identity based on completion of the square on both sides of \eqref{eq:approx-orth-mid-step}: for any matrices $\M_1$, $\M_2$ of appropriate size,
\begin{align*}
\|\M_1 - \M_2\|_F^2 + \delta \|\M_2\|_F^2 &= \|\M_1\|_F^2 + (1 + \delta) \|\M_2\|_F^2 - 2 \langle \M_1, \M_2\rangle  \\ &= \frac{\delta}{1 + \delta} \|\M_1\|_F^2 + \frac{1}{1 + \delta}\left\|\M_1-(1 + \delta) \M_2\right\|^2_F. 
\end{align*}
Using this identity for $\M_1 = \Q\X$ \revtwo{for both sides} and $\M_2 = \Q \tilde{\X}$ on the left and $\M_2 = \Q\Y$ on the right of \eqref{eq:approx-orth-mid-step}, we obtain
\begin{align*}
&\frac{\delta}{1 + \delta} \|\Q \X\|_F^2 + \frac{1}{1 + \delta}\|\Q(\X-(1 + \delta)\tilde{\X})\|^2_F + \delta \| P_{\A^T}^\perp \tilde{\X}\|^2_F \\ &\leq \frac{\delta}{1 + \delta} \|\Q \X\|_F^2 + \frac{1}{1 + \delta}\|\Q(\X-(1 + \delta)\Y)\|^2_F +  \delta \|P_{\A^T}^\perp \Y\|^2_F + 3\varepsilon^2\|\Q(\X-\Y)\|^2_F.
\end{align*}
Canceling common terms, letting $\Y := \X_0/(1+\delta)$, we have
\begin{align*}
\frac{1}{1 + \delta}&\left\|\Q(\X-(1 + \delta)\tilde{\X})\right\|^2_F + \frac{\delta}{(1 + \delta)^2} \left\| P_{\A^T}^\perp (1 + \delta)\tilde{\X}\right\|^2_F \\ &\leq  \frac{1}{1 + \delta}\left\|\Q(\X-\X_0)\right\|^2_F +  \delta \left\|P_{\A^T}^\perp \frac{\X_0}{1+\delta}\right\|^2_F + 3\varepsilon^2\left\|\Q(\X-\frac{\X_0}{1+\delta})\right\|^2_F =: \W.
\end{align*}

To estimate the loss on the uncompressed problem, we use \eqref{eq:triangle} and \eqref{eq:q} with the matrices $\X$, $(1+\delta)\tilde{\X}$ and $\Q$ to get
\begin{align*}
\|\X - (1+\delta) \tilde{\X}\|^2_F &\leq \|\Q(\X- (1+\delta) \tilde{\X})\|^2_F + 2\|P_{\A^T}^\perp(1+\delta) \tilde{\X}\|^2_F + 2\|P_{\A^T}^\perp \X\|^2_F \\
&\le \frac{2(1 + \delta)^2}{\delta} \W + 2\|P_{\A^T}^\perp \X\|^2_F \\
&\le \frac{2(1 + \delta)}{\delta}\|\Q(\X-\X_0)\|^2_F  + \frac{6\varepsilon^2}{\delta}\|\Q((1 + \delta)\X-\X_0)\|^2_F \\
&\quad\quad\quad\quad\quad\quad\quad\quad\quad+  2 \|P_{\A^T}^\perp \X_0\|^2_F + 2\|P_{\A^T}^\perp \X\|^2_F. 
\end{align*}
\revtwo{Now, since $ \|P_{\A^T}^\perp \X_0\|^2_F \leq 2\|P_{\A^T}^\perp (\X-\X_0)\|^2_F + 2\|P_{\A^T}^\perp \X\|^2_F$, and then using $\delta = \lambda/(1 - \eps)^2$ and $\varepsilon \leq 0.5$, we can finally bound $\|\X - (1+\delta) \tilde{\X}\|^2_F$ by}
\begin{align*}
 \left(\frac{2 + 12\eps^2}{\delta} + 2\right)&\|\Q(\X-\X_0)\|^2_F + 4\|P_{\A^T}^\perp(\X-\X_0)\|^2_F + 6\|P_{\A^T}^\perp \X\|^2_F + 12\eps^2\delta \|\Q\X\|_F^2 \\
&\le \left( 4 + \frac{5}{4\lambda} \right) \|\X- \X_0\|^2_F + 6 \|P_{\A^T}^\perp \X \|^2_F + 12 \frac{\varepsilon^2\lambda}{(1 -\eps)^2} \|\X\|_F^2.
\end{align*}
\end{proof}

\begin{remark}\label{rem:apx-orth-discussion} We conclude the section with the discussion  of  Theorem~\ref{thm: apx orthogonal A} result.  We note that
\begin{itemize}
  \item  \revised{Observe that the proofs of Theorem~\ref{thm: apx orthogonal A} and Theorem~\ref{thm:oneside_penalty} above rely on the nonnegativity constraint only to ensure that the solution of the original NMF problem remains feasible for the compressed one.}
    \item Theorem~\ref{thm: apx orthogonal A} shows it is possible to regularize the compressed NMF problem without the projection operator and to find a $(1+\lambda)$-rescaled factors. Note that the rescaling does not affect the learned nonnegative low-rank structure.
\item The property \eqref{eq:approx-orth} $\|P_{\A^T}\M\|^2_F \sim \|\A\M\|^2_F$ is \revtwo{ very different from } the standard geometry preservation properties of the form  $\|\M\|^2_F \sim \|\A\M\|^2_F$, such as Johnson-Lindenstrauss, oblivious subspace embedding, or restricted isometry property (which can be shown to fail for, e.g., random Gaussian matrices $\A \in \R^{k \times m}$ as long as $k < m$). In contrast, the approximate orthogonality property \eqref{eq:approx-orth}  is not hard to satisfy with generic random matrices. For example, an i.i.d.\ Gaussian matrix $\A \in \R^{k \times m}$ with each entry having mean $0$ and variance $\frac{1}{m}$ is $\eps$-approximately orthogonal with probability $1 - 2 \exp(-\eps^2m/8)$ as soon as $k \ge  m\eps^2/4$ (by \cite[Corollary~5.35]{Vershynin11}).
\item While it is easy to guarantee approximate orthogonality with generic matrices $\A$ (not learned from $\X$), the term $P_{\A^T}^\perp \X$  is still the part of the error bound. So, data-oblivious one-sided compression in general is not expected to result in exact recovery even if data matrix $\X$ admits exact nonnegative factorization.
\end{itemize}
\end{remark}

\subsection{Nonnegativity in compression}\label{sec:nonnegativity}

In the further sections, we consider the variations of multiplicative updates algorithm to iteratively minimize our proposed sketched objective functions.
The multiplicative updates algorithm is valid due to the fact that the matrices involved in the iterative process are nonnegative.
Note that this convenient property is destroyed by sketching unless we have that $\A^T\A$ is an element-wise nonnegative matrix. However, the nonnegativity of the sketching matrix is not expected to hold neither for approximately orthonormal random sketches nor for the data-adapted sketching matrices coming from randomized rangefinder algorithm. 

\revtwo{In this section, we show how to overcome this issue: it actually suffices} to add some extra penalty terms taking the form
\begin{equation}\label{eq:sigma-regularizer}
\sigma \|\mathbf{1}_m^T (\X-\U\V^T)\|^2 \quad \text{and/or} \quad \sigma \| (\X-\U\V^T)\mathbf{1}_n \|^2,
\end{equation}
where $\mathbf{1}_m$ is a vector of all ones in $\R^m$.

\begin{corollary}\label{cor:shifted-two-sided}
Suppose $\X$ has an exact nonnegative factorization $\X = \U_0\V_0^T$, where $\U_0 \in \R^{m \times r}$, $\V_0 \in \R^{n \times r}$ \revised{with} $r \le \min\{n,m\}$.
Let $\A_1$ and $\A_2$ be generic random matrices of sizes $k_1 \times m$ and $n \times k_2$, respectively, \revised{with $k_1,k_2 \geq r$}.
If for some $\lambda_1, \lambda_2 > 0$ and $\sigma_1, \sigma_2 \ge 0$
\begin{align}\label{eq:two-sided-2}
 \tilde{\U}, \tilde{\V} \revtwo{\; \in \;}  \argmin_{\U, \V \ge 0} L&(\X-\U\V^T) \\
 &+ \sigma_1 \|\mathbf{1}_m^T (\X-\U\V^T)\|^2 + \sigma_2 \| (\X-\U\V^T)\mathbf{1}_{n}\|^2, \nonumber
\end{align}
where $L(\X-\U\V^T) := \|\A_1 (\X- \U\V^T)\|_F^2 + \|(\X-\U\V^T)\A_2\|_F^2 + \lambda_1\|P_{\X\A_2}^\perp \U\V^T\|_F^2 + \lambda_2\|\U\V^T P_{\X^T\A_1^T}^\perp\|_F^2,$
then $\X = \tilde{\U}\tilde{\V}^T$. 
\end{corollary}
\begin{proof}
Note that
\begin{align*} 
\min_{\U,\V \geq 0} L(\X-\U\V^T) &+ \sigma_1 \|\mathbf{1}_m^T (\X-\U\V^T)\|^2 \\
&\leq \min_{\substack{\U,\V \geq 0 \\ \mathbf{1}_m^T\X = \mathbf{1}_m^T\U\V^T}} L(\X-\U\V^T) + \sigma_1 \|\mathbf{1}_m^T (\X-\U\V^T)\|^2 \\
&= \min_{\substack{\U,\V \geq 0 \\ \mathbf{1}_m^T\X = \mathbf{1}_m^T\U\V^T}} L(\X-\U\V^T),
\end{align*}
and similarly for adding the term $\sigma_2 \| (\X-\U\V^T)\mathbf{1}_n\|^2_F$. Then the statement follows directly from Theorem~\ref{thm:exact-two-sided}.
\end{proof}

When we do not assume that $\X$ \revised{has an exact nonnegative decomposition of dimension $r$}, adding the regularizer of the form \eqref{eq:sigma-regularizer} is still possible under an additional condition, essentially imposing that the column-sums of $\X$ and $\U\V^T$ approximately match if $\U,\V$ are optimal to \eqref{eq:nmf}.

\begin{corollary}\label{thm:shift-one-sided}
Let $\X \in \R_+^{m \times n}$ be a nonnegative matrix and $\A \in \R^{k \times m}$ is a matrix with orthogonal rows. Let $\U_0 \in \R_+^{m \times r}$, $\V_0 \in \R_+^{r \times n}$ give a solution to the original NMF problem \eqref{eq:nmf} and $\X_0 = \U_0\V_0^T$. Additionally assume that $\left\|\mathbf{1}^T (\X_0 - \X)/\mathbf{1}^T \X \right\| \leq \eps < 1$, where $/$ denotes element-wise division and $\mathbf{1} \in \R^m$ is the vector of all ones. If $\lambda,\sigma > 0,$
\begin{equation}\label{eq:shifted-orth}
\tilde{\U},\tilde{\V} \revtwo{\; \in \;}  \argmin_{\U,\V \geq 0} \left \| \A \left( \X- \U\V^T \right) \right \|^2_F +  \lambda \|P_{\A^T}^\perp \U\V^T\|^2_F + \sigma \| \mathbf{1}^T ( \X - \U\V^T) \|^2,
\end{equation}
then 
$\tilde{\X} := \tilde{\U}\tilde{\V}^T$ satisfies 
$$\frac{\|\X - \tilde{\X}\|^2_F}{\|\X\|^2_F} \le c_\lambda \left[ \frac{\| \X - \X_0\|^2_F}{\|\X\|^2_F} + \frac{\|\X- P_{\A^T} \X\|^2_F}{\|\X\|^2_F} + \varepsilon^2\right].$$
For example, if we have that $\lambda, \eps \in (0, 1/2)$ one can bound $\revised{c_\lambda} = \max\{6, 8/\lambda\}.$
\end{corollary}

\begin{proof}
Let $\D \in \R^{n \times n}$ be the diagonal matrix with nonzero entries given by $(\mathbf{1}^T\X/\mathbf{1}^T\X_0)$, and so $\mathbf{1}^T \X_0\D = \mathbf{1}^T \X$. Note that $\X_0\D$ is some (not necessarily optimal) solution to \eqref{eq:shifted-orth} and so we have
\begin{align}\label{eq:shifted-optimality}
\| P_{\A^T}(\X-\tilde{\X})\|^2_F + \lambda \|P_{\A^T}^\perp\tilde{\X}\|^2_F
&\leq \| P_{\A^T}(\X-\tilde{\X})\|^2_F + \lambda \|P_{\A^T}^\perp\tilde{\X}\|^2_F + \sigma \| \mathbf{1}^T ( \X - \tilde{\X}) \|^2 \nonumber\\
&\leq \| P_{\A^T}(\X-\X_0\D)\|^2_F + \lambda \|P_{\A^T}^\perp\X_0\D\|^2_F.
\end{align}
Then, for $c_1 = \max(2/\lambda, 1),$
\begin{align*}
\|\X - \tilde{\X}\|^2_F 
&\overset{\eqref{eq:triangle}}{\leq} \|P_{\A^T}(\X - \tilde{\X})\|^2_F + 2\|P_{\A^T}^\perp \tilde{\X} \|^2_F + 2\|P_{\A^T}^\perp \X \|^2_F\\
&\leq c_1\left(\|P_{\A^T}(\X - \tilde{\X})\|^2_F + \lambda\|P_{\A^T}^\perp \tilde{\X} \|^2_F\right) + 2\|P_{\A^T}^\perp \X \|^2_F\\
&\overset{\eqref{eq:shifted-optimality}}{\leq} c_1 \left( \|P_{\A^T}(\X-\X_0\D)\|^2_F +  \lambda \|P_{\A^T}^\perp \X_0\D\|^2_F \right) + 2\|P_{\A^T}^\perp \X \|^2_F\\
&\leq  c_1\|P_{\A^T}(\X-\X_0\D)\|^2_F +  2\lambda c_1 \|P_{\A^T}^\perp (\X - \X_0\D)\|^2_F + (2\lambda c_1 + 2)\|P_{\A^T}^\perp \X \|^2_F \\
&\overset{c_2 = \max(2\lambda,1)}{\leq}  c_1c_2 \|\X - \X_0\D\|^2_F + (2\lambda c_1 + 2) \|P_{\A^T}^\perp \X \|^2_F\\
&\leq 2c_1c_2 \|\X - \X_0\|^2_F + 2c_1c_2 \|\X_0\|^2_F \|\I-\D \|^2 + (2\lambda c_1 + 2) \|P_{\A^T}^\perp \X \|^2_F\\
& \leq 2c_1c_2(1 +2 \varepsilon^2) \|\X - \X_0\|^2_F + 4c_1c_2\varepsilon^2  \|\X\|^2_F + (2\lambda c_1 + 2) \|P_{\A^T}^\perp \X \|^2_F.
\end{align*}
This completes the proof of Corollary~\ref{thm:shift-one-sided}.
\end{proof}

\revised{
\begin{remark}
    We note that the statement of Corollary~\ref{thm:shift-one-sided} holds with any $k,r \ge 1$. However, in general, $k$ needs to upper-bound $r$ for the error terms to be controllable, in the same way as discussed in Theorem~\ref{thm:tropp} above.
\end{remark}}

\section{Methods that solve compressed problems}\label{sec:solving}

In this section we define iterative methods that solve the formulated optimization problems directly, without referring to the original data or any matrices of the large uncompressed size.

\subsection{General convergence for sketched multiplicative updates}

Multiplicative updates \revtwo{(MU)} has been one of the most popular algorithms for \revtwo{the standard NMF problem \eqref{eq:nmf}} since the introduction in \cite{lee2000algorithms}. 
In this section we show how to modify the classical \revtwo{MU} algorithm for the various objectives we have derived in earlier sections. 

To this end, \revtwo{we first prove a general Theorem~\ref{thm: MU} for \revtwo{MU} for multiple minimization terms (instead of just one as in \cite{lee2000algorithms}) after possible additional linear transformations on the left and on the right \eqref{eq:gen-obj-mu}. Then -- in Corollaries~\ref{cor:solving-two-sided}, \ref{cor:solving-one-sided}, \ref{cor:solving apx orthogonal} -- we demonstrate several special cases of this generalized multiplicative updates scheme that extend the applicability of the MU framework specifically to our proposed compressed NMF problems, as additional linear transformations allow to represent projection and shift terms.} 

\begin{theorem}\label{thm: MU} Consider an objective function in the generic form 
\begin{equation}\label{eq:gen-obj-mu}
\argmin_{\U, \V \ge 0}\frac{1}{2}\sum_{i=1}^s \| \A_i (\X^{(\A)}_i - \U\V^T) \|_F^2 + \frac{1}{2}\sum_{j=1}^t \| (\X^{(\B)}_j - \U\V^T) \B_j \|_F^2,
\end{equation}
where 
all $\X_i^{(\A)}, \X_j^{(\B)} \in \R^{m \times n}$, $\A_i \in \R^{k \times m}$, $\B_j \in \R^{n \times k}$ and $(\U, \V \ge 0)$ means $(\U \in \R_+^{m \times r}, \V \in \R_+^{n \times r}).$ 
If all six matrices 
\begin{align}\label{eq:nonneg-sums}\left\{\U, \V, (\sum_{i=1}^s \A_i^T \A_i), (\sum_{j=1}^t \B_j \B_j^T), (\sum_{i=1}^s \A_i^T \A_i\X^{(\A)}_i), (\sum_{j=1}^t \X^{(\B)}_j\B_j \B_j^T) \right\}
\end{align}
are entry-wise nonnegative, then the objective \eqref{eq:gen-obj-mu}
is nonincreasing under the updates
\begin{align*}
\U &\leftarrow \U \circ \frac{\sum_{i=1}^s \A_i^T \A_i \X^{(\A)}_i \V + \sum_{j=1}^t \X^{(\B)}_j \B_j \B_j^T \V}{\sum_{i=1}^s \A_i^T \A_i \U\V^T  \V + \sum_{j=1}^t \U\V^T \B_j \B_j^T \V},\\
\V &\leftarrow \V \circ \frac{\sum_{i=1}^s (\X^{(\A)}_i)^T \A_i^T \A_i \U + \sum_{j=1}^t \B_j \B_j^T (\X^{(\B)}_j)^T \U}{\sum_{i=1}^s \V \U^T \A_i^T \A_i \U + \sum_{j=1}^t \B_j \B_j^T \V \U^T \U}.
\end{align*}

\end{theorem}

\begin{remark} (Implementation considerations)
Note that we can compute the matrix $\A_i^T \A_i \U \V^T \V$ by the multiplication order given by $\A_i^T ((\A_i \U) (\V^T \V))$.
This order never requires us to store a matrix of size larger than $\max(n,m) \times \max(r,k)$.
Similar procedures can be used for other terms appearing in the updates of Theorem~\ref{thm: MU}.
\end{remark}
The proof of the theorem follows the standard approach that justifies the validity of multiplicative updates algorithm for NMF problem (e.g., \cite{lee2000algorithms}), where nonnegativity assumption ensures its validity on the sketched problem. We remark that giving nonnegativity conditions on the sums \eqref{eq:nonneg-sums} (rather than the stronger conditions on the individual terms) matters so we can put realistic assumption on our regularization terms that include orthogonal projection operators (in Section~\ref{sec:mu-compressed-problems}).

\begin{proof}[Proof of Theorem~\ref{thm: MU}]
 Let us consider the step updating the matrix $\U \in \mathbb{R}^{m\times r}_+$. Let $\u \in \R^{mr}_+$ be the vectorization of $\U$ denoted as $\u = \text{vec}(\U)$.
Define 
\begin{align*}
    &\y^{(1)}_i = \text{vec}(\A_i\X^{(\A)}_i) ,\quad \W^{(1)}_i = (  \V \otimes 
 \A_i) \quad \text{ for } i = 1, \ldots, s;,\\
    &\y^{(2)}_j = \text{vec}(\X^{(B)}_j \B_j), \quad   \W^{(2)}_j = ( \B_j^T \V \otimes \I_{m \times m}) \quad \text{ for } j = 1, \ldots, t,
\end{align*}
where $\otimes$ denotes matrix Kronecker product.
Define $\y \in \R^{k(m+n)}$ to be all of the vectors $\y^{(1)}_i$ and $\y^{(2)}_j$ stacked together vertically. Similarly, $\W \in \R^{(kn + km) \times rm}$ is a stack of all of the matrices $\W^{(1)}_i$ and $\W^{(2)}_j$.

Using the mixed Kronecker matrix-vector product property $(\M_1 \otimes \M_2) \text{vec} (\U) = \text{vec}(\M_2\U\M_1^T)$ that holds for any appropriately sized matrices $\U, \M_1, \M_2$, we can rewrite the objective function \eqref{eq:gen-obj-mu} as
\[
F(\u) = \frac{1}{2} \| \y - \W \u\|^2.
\]

 \noindent\textbf{Step 1: Define quadratic majorizing function.} Consider the function
\begin{equation}\label{eq:g-def}
G(\u',\u) = F(\u) + (\u'-\u)^T \nabla F(\u) + \frac{1}{2}(\u'-\u)^T \K_\u (\u'-\u),
\end{equation}
where the matrix $\K_\u$ is a diagonal matrix with the diagonal $(\W^T \W \u) / \u$, recall that $/$ represents elementwise division.
We claim that $G$ \emph{majorizes} $F$, i.e., $G(\u,\u) = F(\u)$ and $G(\u',\u) \geq F(\u')$.
It is clear that $G(\u,\u) = F(\u)$. We can write
\[
G(\u',\u) - F(\u') = \frac{1}{2} (\u'-\u)^T (\K_\u-
\W^T \W) (\u'-\u)
\]
from the comparison of \eqref{eq:g-def} with the Taylor decomposition for the quadratic function $F(\u')$ at $\u$. So, to check that $G(\u',\u) \geq F(\u')$, it is sufficient to show that the matrix $\K_\u-\W^T \W$ is positive semidefinite. Equivalently, it is sufficient to show that 
\[
\M = (\K_u-\W^T \W) \circ \u\u^T
\]
is positive semidefinite. Indeed, for any vector $\z$ of the appropriate size, $\z^T\M\z = (\z\circ\u)^T(\K_\u - \W^T\W)(\z\circ\u)$, recall that $\circ$ defines element-wise product.

\bigskip

\noindent\textbf{Step 2: Matrix $\M$ is positive semidefinite.}
We will check positive semidefitness of the matrix $\M$ directly, Consider any $\v \in \R^{nk}$. Then
\begin{align*}
\v^T \M \v &= \sum_{ij} \v_i \M_{ij} \v_j \\
&= \sum_{ij} \u_i (\W^T\W)_{ij} \u_j \v_i^2 - \v_i \u_i (\W^T\W)_{ij} \u_j \v_j \\
&= \sum_{ij} (\W^T\W)_{ij} \u_i \u_j (0.5 (\v_i^2 +\v_j^2) - \v_i \v_j) \\
&= \frac{1}{2} \sum_{ij} (\W^T\W)_{ij} \u_i \u_j (\v_i - \v_j)^2.
\end{align*}
Now observe
\begin{align*}
\W^T\W &= \sum_i \W_i^T \W_i + \sum_j \W_j^T \W_j \\
&= \sum_i (\V \otimes \A_i)^T (\V \otimes \A_i) + \sum_j (\B_j^T\V \otimes \I)^T (\B_j^T\V \otimes \I) \\
&= \sum_i (\V^T  \V) \otimes (\A_i^T \A_i) + \sum_j (\V^T \B_j \B_j^T \V) \otimes \I\\
&= (\V^T  \V) \otimes \big( \sum_i \A_i^T \A_i \big) + \V^T \big(  \sum_j \B_j \B_j^T \big) \V \otimes \I.
\end{align*}
Thus, by the entry-wise nonnegativity of $\V$, $\sum_i \A_i^T \A_i$, and $ \sum_j \B_j \B_j^T$, the matrix $\W^T\W$ is also entrywise nonnegative. Since $\U$ is also nonnegative,
\[
\v^T \M \v = \frac{1}{2} \sum_{ij} (\W^T\W)_{ij} \u_i \u_j (\v_i - \v_j)^2 \geq 0.
\]
This shows that $\M$ is positive semidefinite, and therefore, $G$ majorizes $F$.

\bigskip

\noindent\textbf{Step 3: The updates minimize the majorizing function.}
To conclude, observe that
\begin{equation}\label{eq:min-g}
\arg \min_{\u'} G(\u',\u) = \u - \K_\u^{-1} \nabla F(\u) = \u \circ \frac{\W^T \y}{\W^T\W\u}.
\end{equation}
In matrix form, this corresponds exactly to the update for $\U$ given by
\[
\U \circ \frac{\sum_i \A_i^T \A_i \X^{(\A)}_i \V + \sum_j \X^{(\B)}_j \B_j \B_j^T \V}{\sum_i \A_i^T \A_i \U\V^T  \V + \sum_j  \U\V^T \B_j \B_j^T \V}.
\]
From majorization property, we have
\[
F\big(\u \circ \frac{\W^T \y}{\W^T\W\u}\big) \leq G\big(\u \circ \frac{\W^T\y}{\W^T\W\u}, \u\big) \overset{\eqref{eq:min-g}}{\leq} G(\u,\u) = F(\u)
\]
and thus, the iterates do not increase the objective.

The updates for $\V$ also do not increase the objective by a similar argument.
The conditions on  $(\sum_{i=1}^s \A_i^T \A_i\X_i)$ and $ (\sum_{j=1}^t \X_j\B_j \B_j^T)$ ensure that the matrices $\U$ and $\V$ are recursively nonnegative under the updates.
\end{proof}

\subsection{\revtwo{MU} for solving regularized compressed problems}\label{sec:mu-compressed-problems}
Now, we demonstrate how the general framework  \eqref{eq:gen-obj-mu} applies to the compressed problems from Section~\ref{sec:compressed-problems}.

\begin{corollary}[Two-sided updates]\label{cor:solving-two-sided} 
Let $\X \in \R_+^{m \times n}$ be a nonnegative matrix and $\A_1 \in \R^{k \times m}$, $\A_2 \in \R^{n \times k}$ are generic compression matrices such that $\A_1\X\A_2$ is invertible.  
 \revised{Consider} $\Q_1 \in \R^{m \times r}$ and $\Q_2 \in \R^{n \times r}$, the matrices which columns form orthonormal bases of the column space of $\X\A_2$ and the row space of $\A_1\X$ respectively. For any $\lambda_1,\lambda_2 \geq 0$, if
\begin{align*}
    \sigma_1 &\geq \max\big(\left(\A_1^T \A_1\right)_-,  \left(\A_1^T \A_1 + \lambda_1 (\I - \Q_1\Q_1^T)\right)_-\big), \\
    \sigma_2 &\geq \max\big(\left(\A_2 \A_2^T\right)_-,\left(\A_2 \A_2^T + \lambda_2 (\I - \Q_2\Q_2^T)\right)_-\big),
\end{align*}
where $\max(\M_1, \M_2)$ denotes maximum of the entries in all of the matrices $\M_i$, and $(\M)_-$ denotes negative part of the matrix, then the objective
\begin{align*}
\|\A_1 (\X- \U\V^T)\|^2_F &+ \|(\X-\U\V^T)\A_2\|^2_F \nonumber\\ &+ \lambda_1\|P_{\X\A_2}^\perp \U\V^T\|^2_F + \lambda_2\|\U\V^T P_{\X^T\A_1^T}^\perp\|^2_F  \\
 & \quad \quad \quad \quad + \sigma_1 \|\mathbf{1}_m^T (\X-\U\V^T)\|^2 + \sigma_2 \| (\X-\U\V^T)\mathbf{1}_n\|^2, \nonumber
\end{align*}

is nonincreasing under the updates
\begin{align*}
\U &\leftarrow \U \circ \frac{\A_{1,\sigma} \X \V + \X\A_{2,\sigma} \V }{ ( \A_{1,\sigma} + \lambda_1 (\I - \Q_1\Q_1^T)) \U\V^T \V + \U\V^T (\A_{2,\sigma} + \lambda_2 (\I - \Q_2\Q_2^T)) \V }\\
\V &\leftarrow \V \circ \frac{\X^T \A_{1,\sigma} \U + \A_{2,\sigma} \X^T \U}{\V \U^T (\A_{1,\sigma} + \lambda_1 (\I - \Q_1\Q_1^T)) \U + (\A_{2,\sigma} + \lambda_2 (\I - \Q_2\Q_2^T)) \V\U^T \U},\nonumber
\end{align*}
where 
$$
\A_{1,\sigma} := \A_1^T \A_1 + \sigma_1 \mathbf{1}_m\mathbf{1}_m^T  \; \text{ and } \; \A_{2,\sigma} := \A_2 \A_2^T +\sigma_2 \mathbf{1}_n\mathbf{1}_n^T.
$$
Note that Theorem~\ref{thm:exact-two-sided} and Corollary~\ref{cor:shifted-two-sided} claim that the optimal solution for \eqref{eq:comp-shift-nmf} is the optimal solution for the uncompressed NMF problem if $\X$ has exactly nonnegative decomposition of the rank at most $k$.
\end{corollary}

\begin{proof}
Consider the setting of Theorem~\ref{thm: MU}, with the matrices 
\begin{align*}
&\{\A_i\}_{i = 1,2,3} = \{\A_1, \sqrt{\sigma_1}\mathbf{1}_m^T, \sqrt{\lambda_1} P_{\X\A_2}^\perp\} = \{\A_1, \sqrt{\sigma_1}\mathbf{1}_m^T, \sqrt{\lambda_1} (\I - \Q_2\Q_2^T)\},  \\
&\{\B_j\}_{j = 1,2,3} = \{\A_2, \sqrt{\sigma_2} \mathbf{1}_n, \sqrt{\lambda_1} P_{\X^T\A_1^T}^\perp\} = \{\A_2, \sqrt{\sigma_2} \mathbf{1}_n, \sqrt{\lambda_1} (\I - \Q_1\Q_1^T)\}, \\
&\{\X^{(\A)}_i\}_{i = 1,2,3} = \{\X^{(\B)}_j\}_{j = 1,2,3}= \{\X, \X, \mathbf{0}\}. 
\end{align*}
Clearly, we have that the matrices $\X^{(\A)}_i$ and $\X^{(\B)}_i$ are nonnegative.
Then to apply Theorem~\ref{thm: MU}, we must check that $\sum_i \A_i^T \A_i$, $\sum_i \A_i^T \A_i \X^{(\A)}_i$, $\sum_j \B_j \B_j^T$ and $\sum_j \X^{(\B)}_j \B_j \B_j^T$ are entry-wise nonnegative.
First, we calculate
\begin{align*}
\sum_i \A_i^T \A_i &
= \A_1^T\A_1 + \sigma_1 \mathbf{1}_m\mathbf{1}_m^T + \lambda_1(\I - \Q_1\Q_1^T), \\
\sum_j \B_j \B_j^T &
= \A_2\A_2^T + \sigma_2 \mathbf{1}_n\mathbf{1}_n^T + \lambda_2 (\I - \Q_2\Q_2^T).
\end{align*}
Thus to ensure entry-wise nonnegativity of both sums, we need
\begin{align}\label{eq:sigma-1}
    &\sigma_1 \geq \max\big(\left(\A_1^T \A_1 + \lambda_1 (\I - \Q_1\Q_1^T)\right)_-\big), \\
    &\sigma_2 \geq \max\big(\left(\A_2 \A_2^T + \lambda_2 (\I - \Q_2\Q_2^T)\right)_-\big) \nonumber,
\end{align}
that is, $\sigma_1$ and $\sigma_2$ upper-bound maximum elements in the respective matrices.
Similarly, for 
\begin{align*}
\sum_i \A_i^T \A_i \X_i &= (\A_1^T\A_1 + \sigma_1 \mathbf{1}_m\mathbf{1}_m^T)\X, \\
\sum_j \X_j \B_j \B_j^T &= \X(\A_2\A_2^T + \sigma_2 \mathbf{1}_n\mathbf{1}_n^T).
\end{align*}
we need to ensure
\begin{equation}\label{eq:sigma-2}
    \sigma_1 \geq \max\big(\left(\A_1^T \A_1\right)_-\big) \; \text{ and } \; \sigma_2 \geq \max\big(\left(\A_2 \A_2^T\right)_-\big).
    \end{equation}
 With  $\sigma_1$ and $\sigma_2$ large enough to satisfy both \eqref{eq:sigma-1} and \eqref{eq:sigma-2}, we conclude the proof of Corollary~\ref{cor:solving-two-sided} as a special case of Theorem~\ref{thm: MU}.
\end{proof}

\begin{corollary}[One-sided updates for orthogonal $\A$] 
\label{cor:solving-one-sided} If $\X\in\R_+^{m \times n}$ and sketching matrix $\A \in \R^{k \times m}$ has orthogonal rows, $\lambda \in [0,1]$ and the nonnegativity correction term $\sigma \geq \max\big(\left(\A^T\A\right)_-\big)$, then the objective 
\begin{equation}\label{eq:comp-shift-nmf}
\|\A(\X-\U\V^T)\|^2_F +  \lambda \|P_{\A^T}^\perp \U\V^T\|^2_F + \sigma \|\mathbf{1}^T(\X- \U\V^T)\|^2
\end{equation}
with respect to the variables $\U \in \R_+^{m \times r}, \V \in \R_+^{n \times r}$ 
 is nonincreasing under the updates
\begin{align}\label{eq:one-sided-mu-updates}
\U &\leftarrow \U \circ \frac{\A^T\A\X\V + \sigma \mathbf{1}\mathbf{1}^T\X\V}{(1-\lambda)\A^T\A\U\V^T\V + \sigma \mathbf{1}\mathbf{1}^T\U\V^T\V + \lambda\U\V^T\V} \\
\V &\leftarrow \V \circ \frac{\X^T\A^T\A\U + \sigma \X^T\mathbf{1}\mathbf{1}^T\U}{(1-\lambda)\V\U^T\A^T\A\U + \sigma \V\U^T\mathbf{1}\mathbf{1}^T\U + \lambda\V\U^T\U}. \nonumber
\end{align}
Here,  $\mathbf{1} = (1, \ldots, 1) \in \R^m$. Note that Corollary~\ref{thm:shift-one-sided} claims that the optimal solution for \eqref{eq:comp-shift-nmf} results in a good solution for the uncompressed NMF problem as long as the original NMF error $\min\limits_{\U,\V \ge 0}\| \X - \U \V\|^2_F$ and $\|P^\perp_{\A^T} \X\|^2_F$ term are small.
\end{corollary}
\begin{proof}
In the language of Theorem~\ref{thm: MU}, ``$\X_i$" matrices are $\{\X, \mathbf{0}, \X\}$ and  ``$\A_i$'' matrices are $\{\A, \sqrt{\lambda}(\I - \A^T\A), \sqrt{\sigma} \mathbf{1}^T\}$ for $i = 1,2,3$ respectively.
One can see that
$$
\sum_{i = 1}^3 \A_i^T \A_i 
= (1-\lambda) \A^T\A + \lambda \I + \sigma \mathbf{1}\mathbf{1}^T
 \; \text{ and }
\; 
\sum_{i = 1}^3 \A_i^T \A_i \X_i = \A^T \A\X + \sigma \mathbf{1}\mathbf{1}^T\X.$$
These matrices are entry-wise 
nonnegative if $\sigma \geq \max\big(\left(\A^T\A\right)_-\big)$ and $\lambda \in [0,1]$. Theorem~\ref{thm: MU} applies to justify that the objective \eqref{eq:comp-shift-nmf}  is nonincreasing under the updates \eqref{eq:one-sided-mu-updates}. 
\end{proof}

The proof of the next corollary for not necessarily orthogonal sketching matrices is similar to the one above and is omitted for brevity.

\begin{corollary}[One-sided updates for nonorthogonal sketching matrices]\label{cor:solving apx orthogonal} If $\X\in\R_+^{m \times n}$ and $\A \in \R^{k \times m}$ is an arbitrary matrix, $\lambda \in [0,1]$ and the nonnegativity correction term $\sigma \geq \max((\A^T\A)_-)$, then the objective 
\[
\|\A(\X-\U\V^T)\|^2_F +  \lambda \|\U\V^T\|^2_F + \sigma \|\mathbf{1}^T(\X- \U\V^T)\|^2
\]
with respect to the variables $\U \in \R_+^{m \times r}, \V \in \R_+^{n \times r}$
is nonincreasing under the updates
\begin{align*}
\U &\leftarrow \U \circ \frac{\A^T\A\X\V + \sigma \mathbf{1}\mathbf{1}^T\X\V}{\A^T\A\U\V^T\V + \sigma \mathbf{1}\mathbf{1}^T\U\V^T\V + \lambda\U\V^T\V} \\
\V &\leftarrow \V \circ \frac{\X^T\A^T\A\U + \sigma \X^T\mathbf{1}\mathbf{1}^T\U}{\V\U^T\A^T\A\U + \sigma \V\U^T\mathbf{1}\mathbf{1}^T\U + \lambda\V\U^T\U}.
\end{align*}
Here,  $\mathbf{1}$ is a vector of all ones in $\R^m$. 

\end{corollary}

\revtwo{To conclude, we note that all the considered sketched MU updates result in the computational cost of order $(n+m)kr$, in the natural case when $k,r \ll n, m$. At the same time, full MU update takes order $nmr$ operations per iteration. In particular, in the case when $n \sim m$, the quadratic per-iteration cost with respect to the size of the data becomes linear. The iterations of the sketched algorithms are indeed much faster, but they  still could (and usually do) require more iterations until empirical convergence to a set threshold. In Appendix~A we demonstrate this tradeoff empirically. } 
\subsection{Solving compressed problems with projected gradient descent}\label{sec:gd} 
 Multiplicative updates is a popular approach to find nonnegative factorizations due to its simplicity and good convergence properties. 
 However, other standard methods such as alternating nonnegative least
squares, hierarchical least squares, projected gradient descent, among others, can be run on the compressed problems. For comparison and additional example, we will consider 
projected gradient descent (GD) method on compressed data.

For an arbitrary loss function $L(\U,\V)$, nonnegative projected GD can be defined
\begin{align*}
\U &\leftarrow (\U - \alpha \nabla_\U L(\U,\V))_+ \\
\V &\leftarrow (\V - \alpha \nabla_\V L(\U,\V))_+,
\end{align*}
where $\alpha$ is the step size.
For the sake of concreteness, we will give an example of the updates for one of our formulated objective functions.
The projected gradient descent updates for the objective 
$
\|\A(\X-\U\V^T)\|^2_F +  \lambda \|P_{\A^T}^\perp \U\V^T\|^2_F
$
(as in Theorem~\ref{thm:oneside_penalty}) are
\begin{align}\label{eq:proj-gd-updates}
\U &\leftarrow \left(\U - \alpha \left( (1-\lambda)\A^T\A\U\V^T\V + \U\V^T\V - \A^T\A\X\V \right) \right)_+ \\
\V &\leftarrow \left(\V - \alpha \left( (1-\lambda)\V\U^T\A^T\A\U + \V\U^T\U - \X^T\A^T\A\U \right) \right)_+ \nonumber.
\end{align}
We can similarly derive updates for our other objective functions.
A disadvantage of \revised{GD with a constant step size} is that it possesses no guarantee of convergence or even a nonincreasing property. 
\revtwo{We note that by choosing step sizes stepsize small enough, or adaptively, one can derive a gradient descent method with a nonincreasing guarantee, but for the sake of simplicity we do not consider this here.} 
Empirically, we see (Figure~\ref{fig:real-world-convergence} below) that on some datasets projected GD shows competitive performance and that its convergence is indeed not monotonic, unlike the convergence of the regularized compressed MU algorithm. \revtwo{We note that smaller step sizes in GD lead to slower convergence of the method, and the right trade-off can be dataset-dependent.}

\section{Experiments}\label{sec: experiments}
We experiment with three datasets coming from various domains. The 20 Newsgroups dataset (``20News") \cite{20news} is a a standard dataset for text classification and topic modeling tasks. It is a collection of articles divided into a total of $20$ subtopics of the general topics of religion, sales, science, sports, tech and politics. The Olivetti faces (``Faces") \cite{samaria1994parameterisation} is a standard image dataset containing grayscale
facial images. It is often used in the literature as an example for different factorization methods including NMF \cite{zhang2008topology, zhao2019deep}. Finally, we construct a synthetic dataset  with regular random data and nonnegative rank of exactly $20$. Specifically, we let $\U$ and $\V$ be $1000 \times 20$ matrices whose entries are distributed like standard lognormal random variables and define $\X = \U \V^T$. 
The dimensions of the datasets are reported in Table~\ref{table:datasets}.

\begin{table}[h]
\centering
\begin{tabular}{l|l|l}
Dataset   & $n$     & $m$         \\
\hline
Synthetic      & 1000  & 1000     \\
Faces \cite{samaria1994parameterisation} & 400   & 4096    \\
20News \cite{20news} & 11314 & 101322  
\end{tabular}
\caption{Dimensions of all datasets studied.}
\label{table:datasets}
\end{table}

All experiments were run on a cluster with 4 2.6 GHz Intel Skylake CPU cores and a NVIDIA V100 GPU.
Compressed methods were implemented using the JAX library to take advantage of GPU acceleration.
The uncompressed methods that we compare were not implemented to use GPU acceleration since in applications the full data matrix may be too large to store on a GPU.
In our case, the 20News data matrix was too large to store as a dense matrix on our GPU.

All algorithms which we test require initial values for the factor matrices $\U$ and $\V$.
We note that the choice of initialization can significantly influence the convergence and overall behavior of NMF algorithms.
To ensure fairness across experiments, we always initialize the matrices $\U$ and $\V$ randomly with all entries being independent standard log-normal random variables.
We also fix the same random seed across comparisons.

To measure the convergence of the methods, besides interpretability of the topics or images, we use the \emph{relative error metric} $\|\X - \U\V^T\|_F / \|\X\|_F$ and the scale invariant \emph{cosine similarity} metric (normalized dot product metric) between $\U\V^T$ and $\X$, defined as $\langle \X, \U\V^T\rangle/\|\X\|_F\|\U\V^T\|_F$. 

\subsection{Exact recovery from compressed measurements is achievable}

\begin{figure}
    \centering
    \includegraphics[width=.5\textwidth]{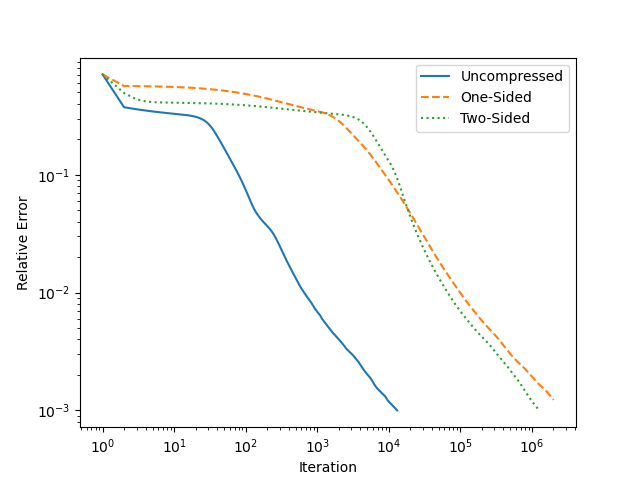}
    \caption{NMF recovery of the synthetic data with MU from full data (Uncompressed); from compressed data with one-sided data-adapted sketches using only $4\%$ of original memory (One-sided), and with two-sided  sketches using  $8\%$ of original memory (Two-sided). In both cases, i.i.d.\ Gaussian sketching matrices are used.} All three methods achieve recovery within $10^{-3}$ relative error.
    \label{fig:synthetic_comparison}
   
\end{figure}

To showcase the exact recovery, we focus on synthetic dataset that has an exact nonnegative factorization with $r = 20$.
In Figure~\ref{fig:synthetic_comparison}, we compare nonnegative factors found by the MU from full data $\X$, from the one-sided data-adapted sketches, and from the two-sided oblivious (i.i.d.\ Gaussian) linear measurements.
We take the target rank $r = 20$ and the sketch size $k = 20$ (so, the sketching matrices have the shape $20 \times 1000$). In this case, the matrix $\X$ has one million elements whereas the total number of elements in $\A$ and $\X\A$ combined is only forty thousand.
This represents a memory ratio of $4 \%$ for the one-sided method.
For the two sided method, the total number of elements in $\A_1$,$\A_2$,$\A_1\X$, and $\X\A_2$ is eighty thousand representing a memory ratio of $8 \%$.

We employ (compressed) MU algorithm as in Corollary~\ref{cor:solving-one-sided} with $\lambda = 0.1$ and Corollary~\ref{cor:solving-two-sided} with $\lambda = 0$. The parameter $\sigma$ is chosen minimal so that we have $\A\A^T \geq -\sigma$, and similarly for the $\A_1$, $\A_2$, $\sigma_1$ and $\sigma_2$ in the two-sided case. 
We see that the algorithm achieves a near exact fit, eventually reducing relative error below $10^{-3}$ which was our stopping criterion in this case.
The one-sided method and the oblivious two-sided methods seem to be converging at a similar rate as the uncompressed method, albeit after a ``burn-in'' phase.

\subsection{Effect of compression size $k$ on recovery}
In Figure~\ref{fig:compression}, we compare the compressed multiplicative updates for different amounts of compression on the same synthetic dataset: the target rank $r = 20$ and the sketch size $k$ varies. 
 \begin{figure}[t!]
        \subfloat[ $\A$ is a random Gaussian matrix]{%
            \includegraphics[width=.5\linewidth]{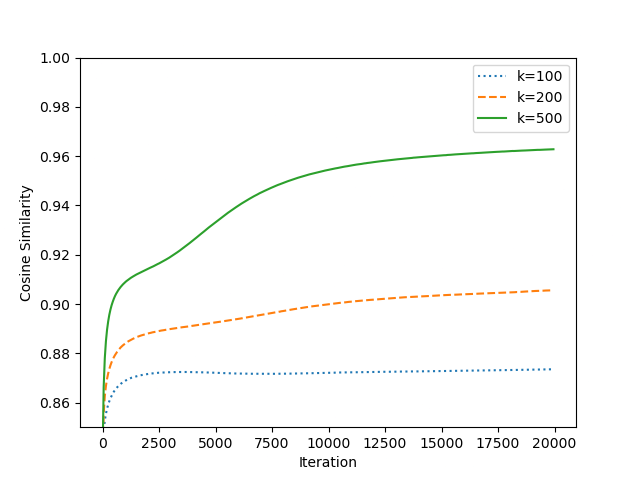}%
            \label{subfig:a}%
        }\hfill
        \subfloat[ $\A$ is a random matrix with orthogonal rows]{%
            \includegraphics[width=.5\linewidth]{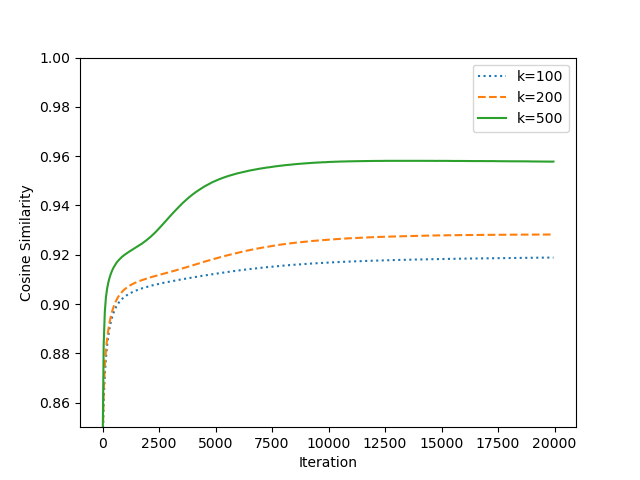}%
            \label{subfig:b}%
        }\\
        
        \vspace*{-.3cm}
        \hfill
        \subfloat[ $\A$ is a data-adapted  matrix]{%
            \includegraphics[width=.5\linewidth]{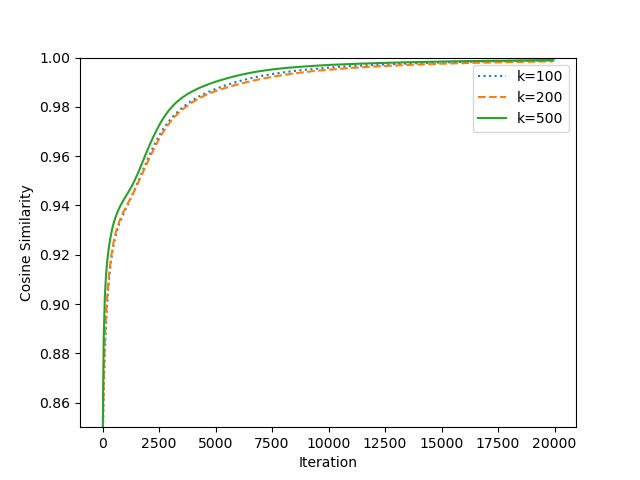}%
            \label{subfig:d}%
        }
           \subfloat[Two-sided compression with Gaussian sketches]{%
            \includegraphics[width=.5\linewidth]{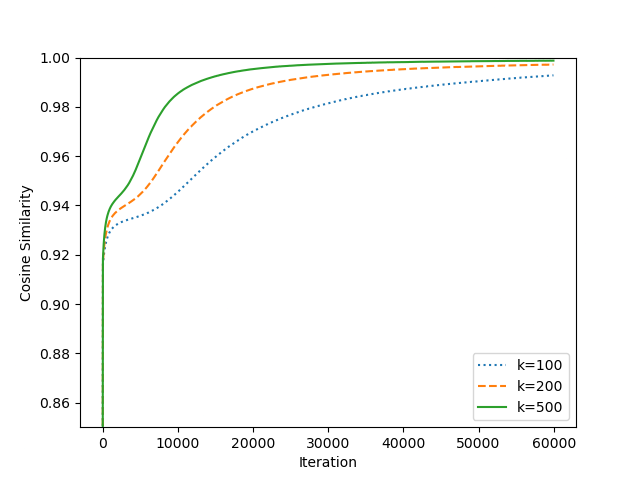}%
            \label{subfig:c}%
        }
        \caption{NMF recovery of the synthetic data with MU. Displays (c,d) show that MU on data-adapted and random two-sided sketched data also tend to the limiting similarity $1.0$. Across all methods, less compression (larger $k$) improves convergence.}
\label{fig:compression}
    \vspace*{-.25cm}
    \end{figure}

Specifically, we choose the compression matrix $\A$ to be (a) a random Gaussian matrix, (b) a random matrix with orthogonal rows, or (c) via the randomized SVD procedure (\ref{thm:tropp}). We employ the compressed MU algorithm as in Corollaries~\ref{cor:solving-two-sided}, \ref{cor:solving-one-sided}, and~\ref{cor:solving apx orthogonal} respectively.
For the two-sided method (d), we choose both compression matrices to be random Gaussian matrices and $\lambda = 0$. For the one-sided compression, we take $\lambda = 0.1$. Here, we report cosine similarity loss. Synthetic data has exactly low nonnegative rank $k \le 20$. For reference, MU on the uncompressed data achieve limiting similarity very close to $1$.

Our the theoretical results suggest the possibility of exact recovery from the compressed data in two cases: (1) if two-sided compression was employed (Theorem~\ref{thm:exact-two-sided}), or (2) if the compression matrix is such that $P_{\A}^\perp \X$ is zero, for example, if the compression matrix is learned via the randomized rangefinder algorithm (Corollary~\ref{cor:exact-one-sided}). Correspondingly, the one-sided data-adapted compression and two-sided random compression attain exact reconstruction at different rates depending on the amount of compression (Figure~\ref{fig:compression} (c,d)). Two-sided compression requires twice more memory for the same sketch size but works with generic (data-oblivious) measurements that can be advantageous for various reasons as they do not require access to data or another pass over the data. 
    
\revised{For the most compact one-sided data-oblivious measurements, Theorems~\ref{thm: apx orthogonal A} and \ref{thm:oneside_penalty} --- corresponding to the measurements of Figure~\ref{fig:compression} (a,b) respectively --- imply that the additional error might appear as a result of compression (from the term with $\|P^\perp_{\A^T} \X\|^2_F$). Here, we show that on an ``easy" synthetic problem oblivious one-sided measurements can be also used: the compressed MU algorithm results in a good, although not perfect, recovery (Figure~\ref{fig:compression} (a,b)). The amount of additional loss depends on the amount of compression: indeed, for larger $k$ the dimension of the row-space of $\A$ is larger and the operator $P^\perp_{\A^T}$ projects the data into the smaller space, leading to smaller additional losses in the bounds such as \eqref{eq:one-sided-loss} and \eqref{eq:approx-orth-target-error}. }

\revised{In summary, we observe that less compression leads to a faster or better convergence performance in all the scenarios of Figure~\ref{fig:compression}.}

\subsection{Effect of regularization parameter $\lambda$ on recovery} 

We have chosen the regularization parameter $\lambda=0.1$ for the one-sided experiments above.
Here, we demonstrate that it is important empirically, as well as theoretically, to add nonzero regularization term in the one-sided compression case. In Figure~\ref{fig:lambda}, we consider the compressed MU algorithm from one-sided data-adapted measurements for the 20News data with $k = 100$ and $r = 20$.
We can see that regularization can have a beneficial effect on performance and $\lambda = 0$ compromises the convergence. At the same time, too large $\lambda$ could slow down the convergence or result in slightly worse limiting loss. 
\begin{figure}
    \centering
    \includegraphics[width=.45\textwidth]{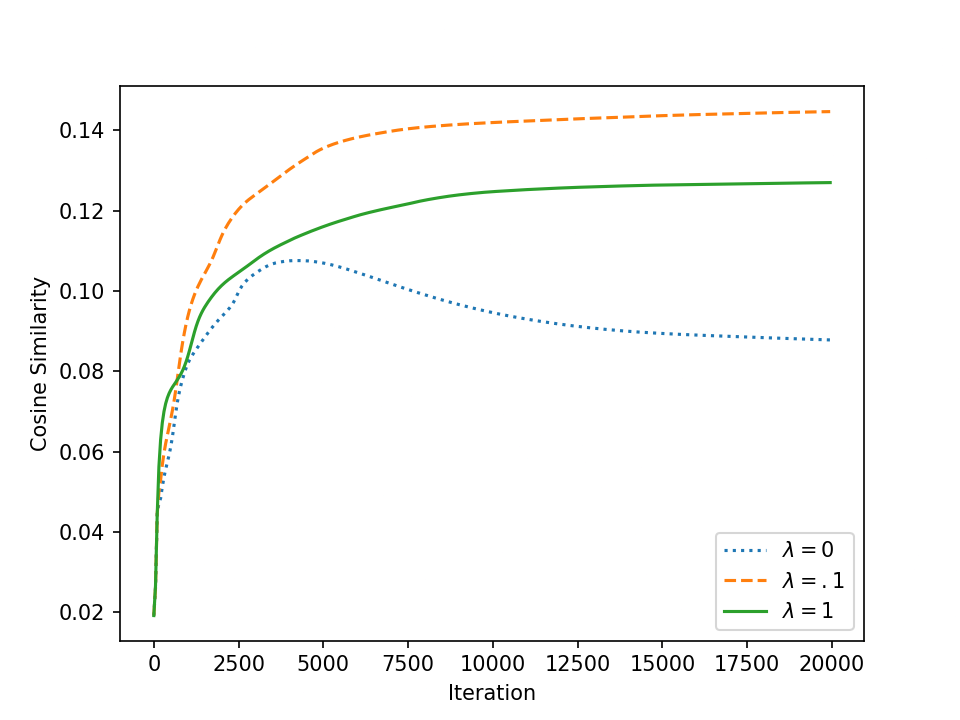}
    
    \caption{Effect of regularization parameter $\lambda$ on the MU algorithm \eqref{eq:one-sided-mu-updates}. 20News dataset compressed with data-adapted one-sided measurements, $\sigma$ is chosen minimal so that we have $\A\A^T \geq -\sigma$. The absence of regularization compromises convergence and too strong regularization results in a higher loss.}
    \label{fig:lambda}
        \vspace*{-.25cm}
\end{figure}

\subsection{Recovery methods performance comparison on real-world data} \label{sec:exp-comp}
We compare our proposed methods with a ``NMF with random projections" method proposed in \cite{wang2010efficient} and in \cite{tepper2016compressed}.
These works adapted the updates of \cite{ding2008convex} to the compressed NMF setting resulting in the updates:
\begin{equation}\label{eq: wangli}
\begin{aligned}
\U &\leftarrow \U \circ \sqrt{\frac{(\Y_2 \A_2^T \V)_+ + (\U\V^T \A_2 \A_2^T \V)_-}{(\Y_2 \A_2^T \V)_- + (\V^T \A_2 \A_2^T \V)_+}}
\end{aligned}
\end{equation}
and similar for $\V$, see equations (3.28) and (3.29) of \cite{wang2010efficient}.
The work of \cite{wang2010efficient} propose to use the updates \eqref{eq: wangli} where $\A_1$ and $\A_2$ are chosen to be Gaussian matrices.
In this case we denote these updates (WL) in the legends.
The work of \cite{tepper2016compressed} also proposed using the randomized rangefinder procedure \cite{halko2011finding} as in our Corollary~\ref{cor:exact-one-sided} to choose the matrices $\A_1$ and $\A_2$.
In this case we denote this method (TS) in the legends.

For our two-sided (MU) method we solve a simplified version of the two-sided objective \eqref{eq:two-sided-2} with $\lambda = 0$. We take $\A_1$ and $\A_2$ to be random Gaussian matrices in the oblivious case or according to the randomized rangefinder procedure as in Corollary~\ref{cor:exact-one-sided} in the adaptive case. For the data-adapted one-sided (MU) method, we take $\lambda = 0.1$ and solve the problem \eqref{eq:comp-shift-nmf}. Then, we include the recovery via projected gradient descent (GD), as described in Section~\ref{sec:gd} with a step size of $\alpha = .001$. 

\begin{figure}[t!]
       \subfloat[Faces; random oblivious compression]{
            \includegraphics[width=.48\linewidth]{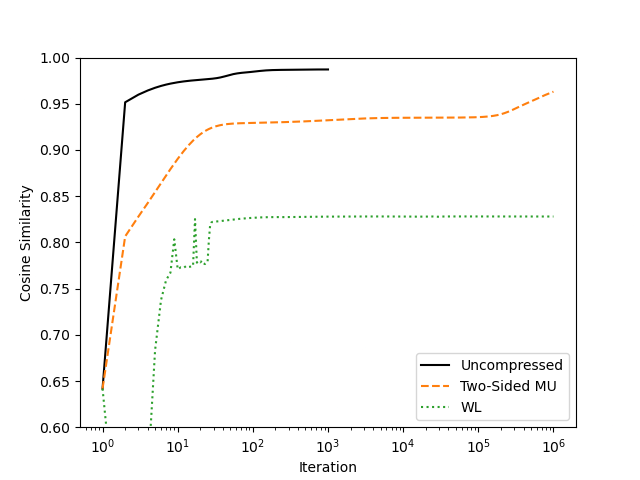}
        }\hfill
        \subfloat[Faces; data-adapted compression]{
            \includegraphics[width=.48\linewidth]{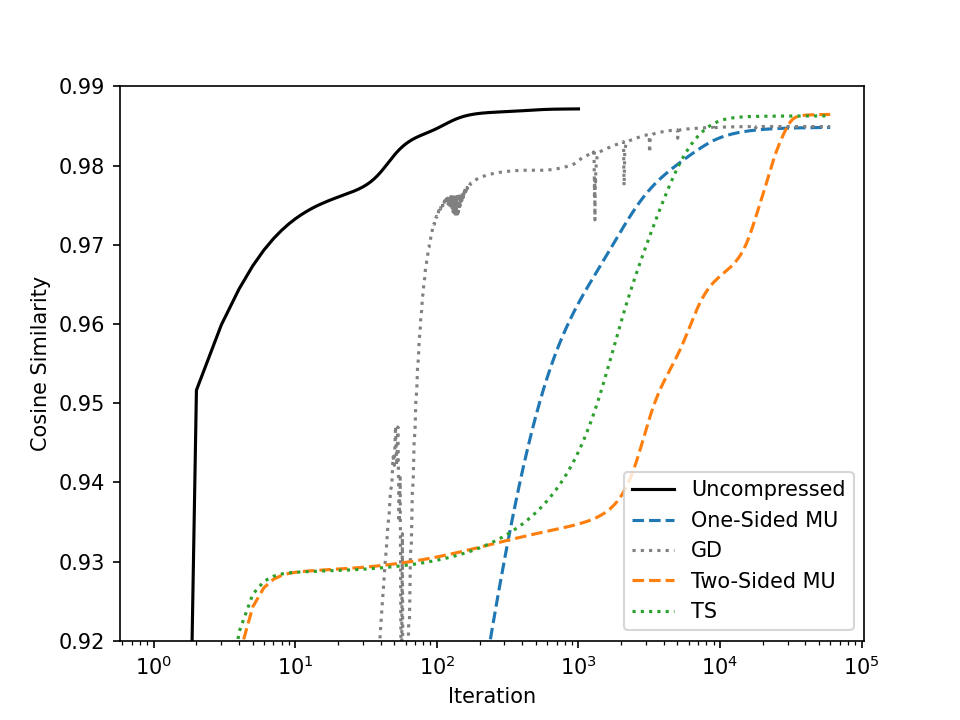}
        }\\
        \subfloat[20News; random oblivious compression]{
            \includegraphics[width=.48\linewidth]{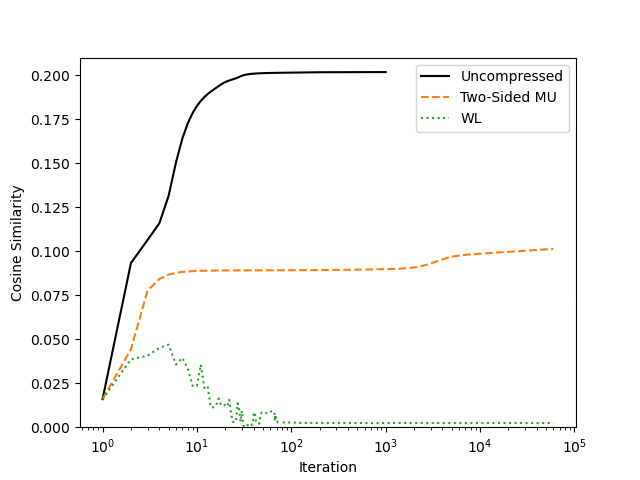}
        }\hfill
        \subfloat[20News; data-adapted compression]{
            \includegraphics[width=.48\linewidth]{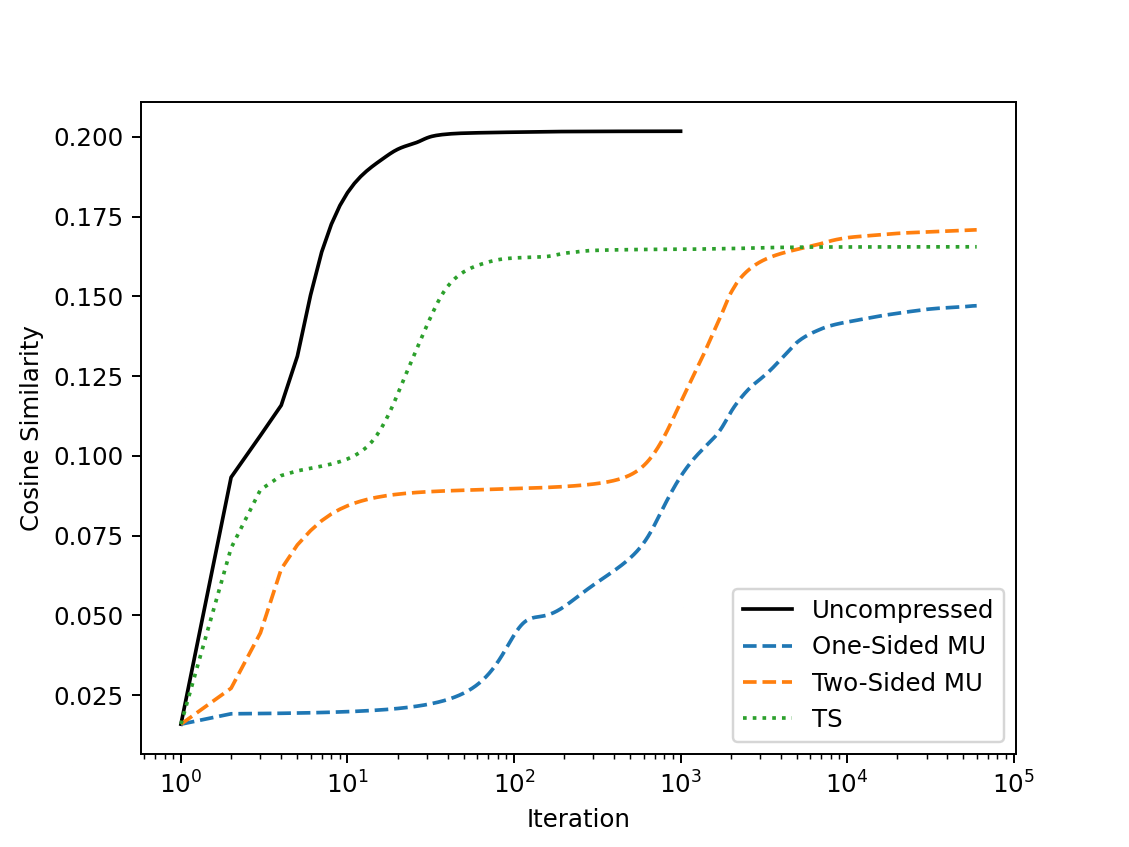}
        }
        \caption{(a,c): recovery from random Gaussian measurements, averaged over $5$ runs; our two-sided MU methods lead to better convergence than WL \cite{wang2010efficient}.
        (b,d): data-adapted methods with sketching matrices obtained with the randomized rangefinder algorithm (like in Corollary~\ref{cor:exact-one-sided});  
        our two-sided \revtwo{MU} perform slightly better than TS \cite{tepper2016compressed} after enough iterations. 
        }
        \label{fig:real-world-convergence}
       \vspace*{-.25cm}
    \end{figure}

\medskip

In Figure~\ref{fig:real-world-convergence} (a,b), we study the performance of compressed nonnegative matrix factorization for the Faces dataset with target rank $r = 6$ and sketch size $k = 20$. 
For the memory, the data matrix $\X$ contains $1638400$ elements whereas the total number of elements in $\A$ and $\X\A$ is $89920$, representing a memory ratio of approximately $5\%$ in the one-sided case and $10\%$ in the two-sided compression.  \emph{On the Faces dataset, our methods, especially data-adapted versions, attain almost the quality of MU on the uncompressed data while using only $5\%$ of the memory ($10\%$ in the two-sided case).}

In Figure~\ref{fig:real-world-convergence} (c,d), we study the performance of compressed nonnegative matrix factorization for the 20News dataset with target rank $r = 20$ and sketch size $k = 100$. 20News is a ``hard dataset" for NMF -- we can see that even in a full uncompressed dataset NMF achieves only $0.2$ cosine similarity (however, this similarity can be enough to do a meaningful job for topic modeling applications) and our compressed MU from data-adapted measurements achieve higher than $0.17$ cosine similarity while using only $2\%$ of memory required for the uncompressed MU with the one-sided compression. Indeed, the number of elements in $\X$ (including zeros) is $1146357108$.
The total number of elements in $\Y_1$, $\Y_2$, $\mathbf{1}^T\X$, $\X\mathbf{1}$, $\A_1$, $\A_2$, $\U$, and $\V$ is $25005192$.
\emph{On the 20News dataset, our compressed MU from data-adapted measurements attain $85\%$ of the similarity using only $2\%$ of the memory ($4\%$ in the two-sided case) compared to the uncompressed NMF problem.}

\medskip

Since it might be infeasible to even run an uncompressed problem from a memory perspective, we do not specifically focus on time performance here. However, we note that while it typically requires less iterations for the uncompressed NMF to learn the factorization, the iterations themselves are considerably slower. In Figure~\ref{fig:real-world-convergence} (c,d) it would take several hours to run the uncompressed experiment until $60,000$ iterations, while the other methods take at most several minutes, so we show only the first $10^3$ iterations for uncompressed NMF. \emph{For the Faces dataset, it took $16$ sec to run compressed MU and $12$ sec to run GD until $60,000$ iterations, and we can see that $10$ times less iterations would have been enough for approximately same quality. The uncompressed method took $7$ sec for the plotted $10^3$ iterations, so at least $7$ min would be required to run it for $60,000$ iterations. }

\revised{See Appendix A for additional information about the expreriments from 
 this section, including the time and resulting cosine similarity performance of each method.}

\subsection{Recovery methods produce interpretable decompositions on real-world data}
An important property of nonnegative low-rank factorizations is getting interpretable components. In Figure~\ref{fig:faces}, we briefly demonstrate that the components learned from the compressed data are also interpretable. That is, we show the columns of the fitted $\V$ matrix reshaped back to $64 \times 64$ images in the same setup as in Figure~\ref{fig:real-world-convergence} (b) for the one-sided data-adapted measurements.
\begin{figure}[t!]
        \subfloat[Compressed: MU \eqref{eq:one-sided-mu-updates}]{%
            \includegraphics[width=.32\linewidth]{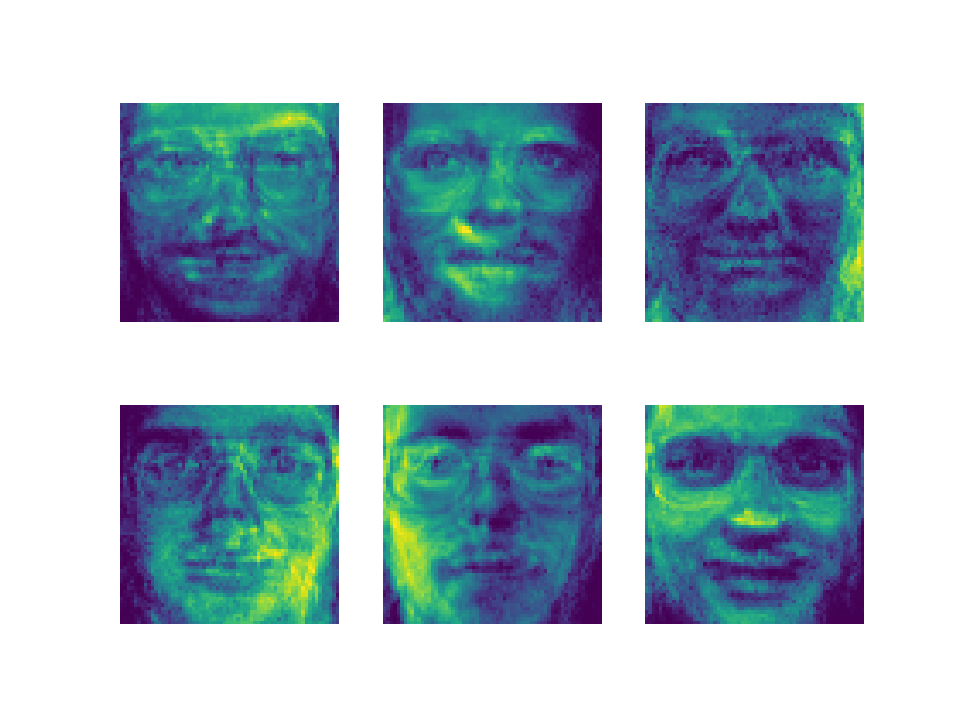}%
        }\hfill
        \subfloat[Compressed: GD \eqref{eq:proj-gd-updates}]{%
            \includegraphics[width=.32\linewidth]{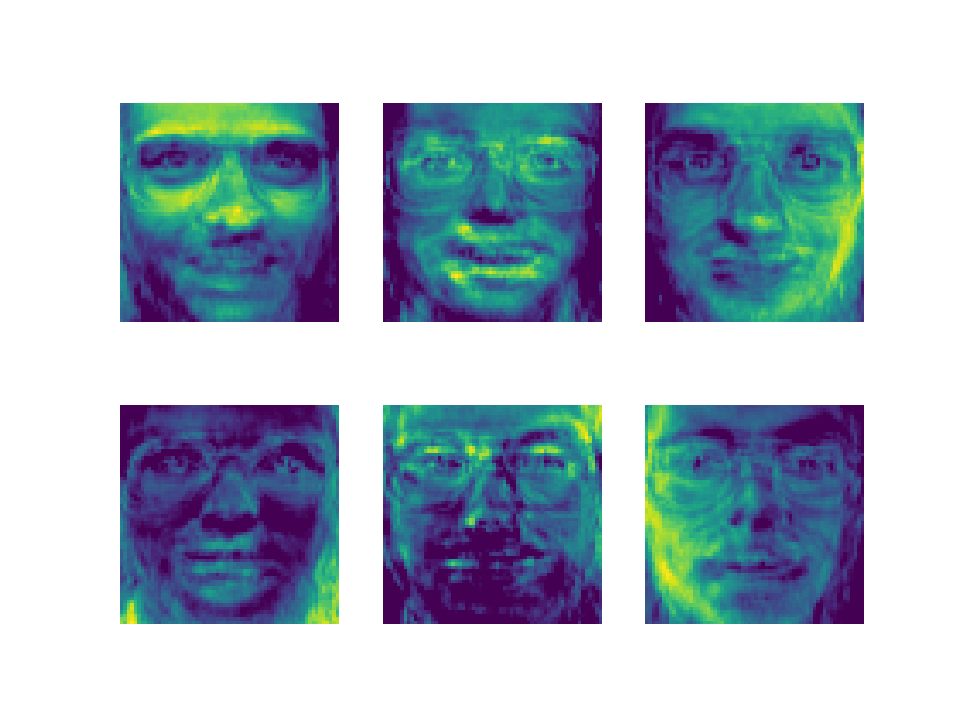}%
        }
        \subfloat[Full data: MU]{%
            \includegraphics[width=.32\linewidth]{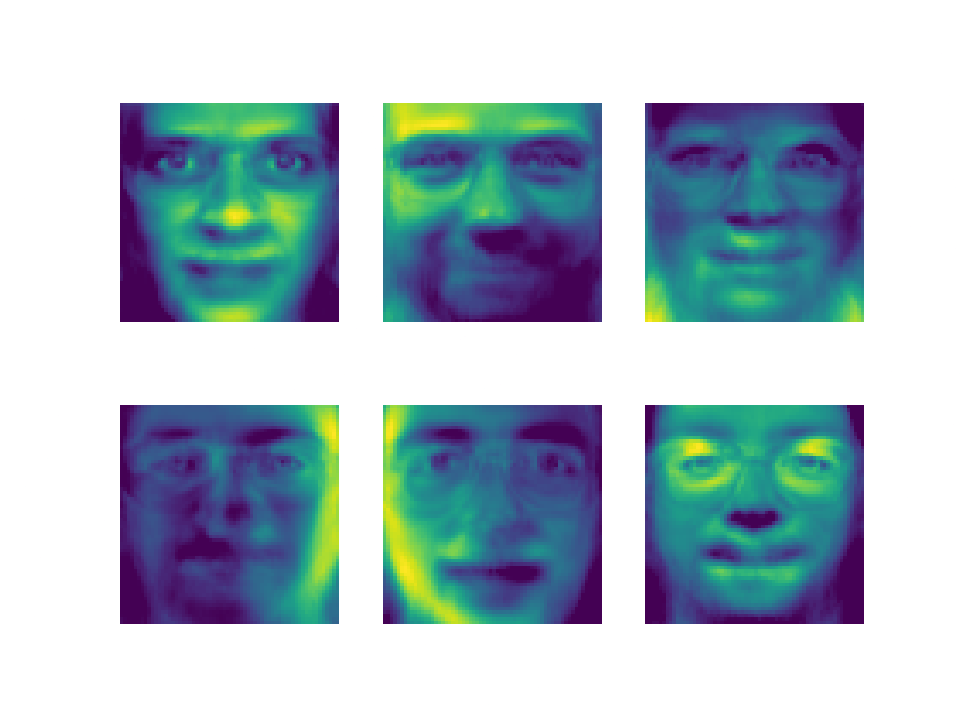}%
        }
        \caption{Six ``representative" faces from the Faces dataset learned from the compressed dataset of the size $\sim 5\%$ of initial data. Data-adapted compression matrix $\A$ is used.}
        \label{fig:faces}
         \vspace*{-.25cm}

    \end{figure}

\section{Conclusion and future directions}\label{sec:future}
In this paper, we propose several formulations of the NMF problem that (a) work using only compressed initial data, where the compression is done with linear maps (sketches) that access initial data only once or twice; (b) have optimal solutions that are provably compatible with the optimal solutions of the uncompressed NMF problem; (c) are supplemented with memory-efficient algorithms that solve the compressed problems without returning to the initial large data or forming any matrices of the original size. The convergence of these algorithms is proved in a standard for the NMF-related algorithms form, that is, showing that the objective function (that we provably connect to the original uncompressed objective) does not increase under the updates. We supplement the theoretical results with the experiments showing comparable nonnegative factorization performance using only $\sim 5\%$ of initial data, on artificial, text and image datasets.

\vspace*{0.2cm}
There are multiple venues of future work stemming from our approach. For the two-sided measurements, we currently do not have a theorem that holds for the data matrices with approximately low nonnegative rank, like in the one-sided case. The experimental evidence clearly suggests similar more general results should hold in the two-sided case. Also, it would be interesting to explain theoretically why the two-sided compression is less sensitive to the regularization (in practice, we take $\lambda = 0$ in the two-sided experiments, which significantly simplifies the updates from Corollary~\ref{cor:solving-two-sided}).  

An important challenge is to meaningfully transfer the obtained theoretical connection between the optimal values of the standard and compressed NMF to the decompositions that we learn in practice. The obstacle is in poor guarantees for the convergence of the NMF learning methods. Recently, in \cite{lyu2020block}, it was shown that a regularized and component-wise bounded version of MU indeed converges to a first-order optimal point with a controlled number of iterations. It would be reasonable to consider an analogue of this version of MU for the compressed objective and tighten the algorithmic connection between the full and compressed problems through it.

Then, it is important to study the scalable versions of other nonnegative factorization algorithms besides multiplicative updates and projected gradient descent. We focused on multiplicative updates because of their relative popularity and simplicity, but perhaps other methods may be better adapted to sketched problems. A related question is to get theoretical guarantees for the methods proposed in \cite{tepper2016compressed} that empirically show comparable performance and typically faster convergence than our proposed algorithms. 

Further, it is natural and meaningful to extend the framework to the compressed versions of high-order (tensor) problems. It has been recently shown \cite{kassab2020sparseness, vendrow2021neural} that nonnegative tensor factorization result in more interpretable decomposition for naturally high-order data, such as temporal or multi-agent, than their matrix NMF counterparts. At the same time, the tensor methods are even more computationally demanding and would benefit from scalable approximations. Scalable versions of other NMF-based algorithms such as semi-supervised versions \cite{semisupervised} or sparse NMF \cite{kassab2020sparseness, hoyer2004non} are also of interest. \revtwo{Finally, while in this work we focus solely on linear sketching techniques, it is interesting to investigate both theoretical and empirical properties of nonlinear dimension reduction for the NMF and nonnegative tensor factorization problems.}

\bibliographystyle{plain}
\bibliography{references}

\appendix

\section{Additional numerical evidence: time comparison}
Here, we include additional summarizing information for the experiments from Section~\ref{sec:exp-comp}.
\begin{table}[h]
\centering
\begin{tabular}{p{2.2cm}|p{1.7cm}|p{1.3cm}|p{1.6cm}|p{1.9cm}|p{1.5cm}}
Method &  Oblivious\revtwo{/} \newline Adapted & Memory needed &Time per \newline iteration \newline \revtwo{(ms)} & Similarity \newline after $10^3$ \newline iterations & Final \newline similarity  \\
\hline
\hline
Uncompressed  & & 100\% & 6.79  & \textbf{.9872} & N/A \\\hline
WL & O & 10\% &  0.26  & .8279 & .8280 \\\hline
TS & A & 10\% & 0.27  & .9439 & \textbf{.9863}  \\\hline
MU 2-sided & O & 10\% & 0.23  & .9322 & .9631  \\\hline
MU 1-sided & A & 5\% & 0.27  & .9626 & \textbf{.9848} \\\hline
MU 2-sided & A & 10\% & 0.23  & .9347 & \textbf{.9865} \\\hline
GD & A  & 5\%& 0.2  & \textbf{.9806} & \textbf{.9849} \\\hline
\label{table2}
\end{tabular}
\caption{Summary of results for the Faces dataset. Oblivious measurements (O) advantageous as they can be applied to any data, data-adapted (A) typically result in better performance. All oblivious methods were run until $10^6$ iterations total and all data-\revtwo{adapted} methods until $6*10^4$ iterations total; final similarity is reported at the last iteration. Time is reported in microseconds. Similarity above $98\%$ is highlighted. }
\label{table:faces}
\end{table}

\begin{table}[h]
\centering
\begin{tabular}{p{2.2cm}|p{1.7cm}|p{1.3cm}|p{1.6cm}|p{1.9cm}|p{1.5cm}}
Method &  Oblivious\revtwo{/} \newline Adapted & Memory needed &Time per \newline iteration \newline \revtwo{(ms)} & Similarity \newline after $10^3$ \newline iterations & Final \newline similarity  \\
\hline
\hline
Uncompressed  & & 100\% & 50.6  & \textbf{.2018} & N/A \\\hline
WL & O & 4\% &  0.73  & .0024 & .0024 \\\hline
TS & A & 4\% & 0.73  & .1648 & .1655  \\\hline
MU 2-sided & O & 4\% & 1.5  & .0898  & .1013  \\\hline
MU 1-sided & A & 2\% & 1.84  & .0935  & .1471 \\\hline
MU 2-sided & A & 4\% & 1.5  & .1167 & \textbf{.1709} \\\hline
\label{table3}
\end{tabular}
\caption{Summary of results for the 20News dataset. All compressive methods were run until $6*10^4$ iterations total; final similarity is reported at the last iteration. Time is reported in milliseconds. Similarity above $17\%$ is highlighted (it is a hard dataset, but this similarity can be still enough for topic modeling tasks). }
\label{table:20news}
\end{table}

\end{document}